\documentclass[12pt]{article}
\usepackage{amsfonts}
\usepackage{amsmath}
\usepackage{amsthm}
\usepackage[utf8]{inputenc}
\usepackage[nottoc]{tocbibind}
\usepackage{titlesec}
\usepackage{hyperref}
\usepackage{enumitem}
\usepackage{titlesec}
\usepackage{graphicx} 
\usepackage{indentfirst}
\usepackage{geometry}
\usepackage{mathtools}
\usepackage{yfonts}
\usepackage{tikz-cd}
\usepackage{comment}

\title{Valuative invariants for linearized line bundles on a spherical variety}
\author{Chenxi Yin}
\date{}

\newtheorem{thm}{Theorem}[section]
\newtheorem{remark}{Remark}[section]
\newtheorem{lem}{Lemma}[section]
\newtheorem{prop}{Proposition}[section]
\newtheorem{cor}{Corollary}[section]
\theoremstyle{definition}
\newtheorem{definition}{Definition}[section]
\theoremstyle{definition}
\newtheorem{exm}{Example}[section]
\newtheorem*{ac}{Acknowledgement}

\newcommand{\F}{\mathcal{F}}

\begin{document}

\maketitle

\begin{abstract}
    We give explicit formulas for various valuative invariants of linearized ample line bundles on projective spherical varieties. Then we show how to use these formulas to study $g$-solitons on a Fano spherical variety. As a corollary, we show that for a Fano horospherical manifold $X$, the corresponding cone $(K_{X})^{\times}$ always admits a Calabi-Yau cone metric.
\end{abstract}

\section{Introduction}

Let $X$ be a Fano manifold. In the quest of Kähler–Einstein metrics, the so called $\alpha-$invariant was introduced in \cite{tian1987kahler}. The $\alpha$-invariant plays a crucial role in \cite{tianyau1987kahler} in determining which projective surfaces admit Kähler–Einstein metrics.

It is explained in \cite[Theorem A.3]{cheltsov2008log} that the $\alpha-$invariant can be defined as:
\[\alpha(X,-K_{X}) := \inf\{\mathrm{lct}(X,D)|D \text{ is an effective } \mathbb{Q-}\text{divisor on X and } D\sim_{\mathbb{Q}}-K_{X}\} \]
where:
\[\mathrm{lct}(X,D) := \inf_{E}\frac{A_{X}(E)}{\mathrm{ord}_{E}(D)}=\inf_{v}\frac{A_{X}(v)}{\mathrm{ord}_{v}(D)}\]
Here $A_{X}$ is the log discrepancy function, $E$ runs through all the prime divisors over $X$, meaning that there is a proper birational morphism $\mu:Y\rightarrow X$ from a normal $Y$, with $E$ a prime divisor on $Y$, whereas $v$ runs through all the non-trivial valuations with finite log discrepancy.  
A famous theorem of Tian says that:
\begin{thm}\cite[Theorem 2.1]{tian1987kahler}
    Let $X$ be a Fano manifold. If $\alpha(X,-K_{X}) > \frac{\mathrm{dim}(X)}{\mathrm{dim}(X)+1}$, then we have a Kähler–Einstein metric on $X$.
\end{thm}

We see that the $\alpha-$invariant is a valuative invariant and closely related to singularities. This suggests there should exist a theorem relating the existence of a Kähler–Einstein metric on a Fano manifold to the analysis of singularity via the valuative theory. This program was achieved after significant efforts (see for example \cite{xu2020positivity,li2022g,liu2022finite}). Several valuative invariants have been introduced in this context. In particular, the $\delta-$invariant was defined in \cite[Definition 0.2]{fujita2018k} and the $\beta-$invariant was defined in \cite[Definition 3.5] {Li2017k} and \cite[Definition 1.3]{fujita2019valuative}. It is worth mentioning that the relation between the existence of Kähler–Einstein metrics and singularities also appears in the pioneering work of Nadel \cite{nadel1990multiplier}, where multiplier ideal sheaves were introduced to capture this connection.

Another widely used approach is the theory of K-stability (or Ding-stability). To define these notions of stability, one introduces degenerations of the manifold called test configurations, first considered by Tian in a less general form \cite{TianKEpositive1997}, and later developed in full generality by Donaldson \cite[Definition 2.1.1]{Donaldsonstability2002}. Then one can define numerical invariants such as Futaki invariants and Ding invariants for test configurations as obstructions to the existence of Kähler–Einstein metrics. To consider all the test configurations together paves the way to the definitions of K-stability and Ding-stability. The direction from the existence of a Kähler–Einstein metric to the K-polystability of a $\mathbb{Q}-$Fano variety is proved in \cite[Theorem 1.1]{berman2016k}. The other direction is carried out by \cite{chen2015kahler1,chen2015kahler2,MR3264768,tian2015k}. The equivalence between K-polystability and Ding-polystability for $\mathbb{Q}-$Fano varieties is proved in \cite[Corollary 3.4]{fujita2019valuative}.

The two points of view mentioned in the preceding two paragraphs rely on seemingly different algebraic theories, but they are actually closely related and can be unified by the language of filtrations. To exchange information between the two aspects, one uses the so called special test configurations (Definition \ref{deftf}).

Theories mentioned above are abstract in nature. It is usually difficult to say if a Fano manifold admits Kähler–Einstein metrics, or equivalently, is Ding (K)-polystable or not. It is seldom an easy task to compute the valuative invariants. But the task becomes easier when a group acts on the manifold. Works in this direction include \cite{wang2004kahler,delcroix2015log,delcroix2020k,blum2020thresholds,golota2020delta}, to name just a few.

The purpose of this paper is to study valuative invariants for linearized ample line bundles on spherical varieties. Then we focus on $\mathbb{Q}-$Fano case and use valuative information to study Ding-stability as well as the existence of different canonical metrics. 

Spherical varieties are varieties with large symmetry. To be more precise, let $G$ be a connected reductive complex algebraic group and $B$ a Borel subgroup. Let $X$ be a normal $G-$variety. We say that $X$ is spherical if it contains an open $B-$orbit. The classification theory of spherical $G-$varieties, known as the Luna–Vust theory, was developed in \cite{luna1983plongements}. Toric varieties and flag varieties are two important subclasses of spherical varieties.

 There is a lot of combinatorial information associated to a spherical variety, the most important one being the valuation cone $\mathcal{V}$, which is the set of $G-$invariant $\mathbb{Q}-$valued valuations. Usually the set of valuations for a variety is hard to describe, but $\mathcal{V}$ is rather explicit and has a nice structure. It is actually a polyhedral convex cone inside a $\mathbb{Q}-$vector space $N\otimes \mathbb{Q}$. This nice property allows us to do computations and to have nice conclusions. Sometimes we also use the closure of it inside $N\otimes \mathbb{R}$, which we denote as $\mathcal{V}_{\mathbb{R}}$. When we have a $G$-linearized ample line bundle on a $G-$variety, we can define the moment polytope $\Delta^{+}$ (Definition \ref{moment polytope}), which encodes the most important combinatorial information about a $G-$linearized line bundle. We can also define a polytope $\Delta$ for every $B-$eigenvector $s\in H^{0}(X,L)$ with eigenvalue $\chi$. More detailed information about spherical varieties and $G$-linearized line bundles are given in Section \ref{section 4}. We have $\Delta^{+} = \Delta + \chi$. In this paper, we mainly investigate spherical varieties with a $\mathbb{Q}-$Cartier canonical divisor.

Our first result is a generalization of several results in \cite{delcroix2015log,blum2020thresholds,golota2020delta}. After introducing different valuative invariants in Sections \ref{section 2} and \ref{section 3}, we can explicitly compute:

\begin{thm}
Let $G$ be a connected reductive algebraic group and $X$ a spherical $G-$variety with $\mathbb{Q}$-Cartier $K_{X}$. Let $L$ be a $G$-linearized ample line bundle on $X$. Then we can compute the $\delta^{(p)}$-invariant and $\alpha$-invariant with $G-$action as
\begin{align*}
 \delta^{(p)}_{G}(X,L) &= \min_{v\in E} \frac{h_{\mathcal{C}}(v)}{\left(\frac{\int_{\Delta}P(x)(\langle x,v\rangle +l_{s}(v))^{p}dx}{\int_{\Delta}P(x)dx}\right)^{1/p}},\\
\alpha_{G}(X,L) &= \min_{v\in E} \min_{m\in \Delta}\frac{h_{\mathcal{C}}(v)}{\langle m,v\rangle + l_{s}(v)},
\end{align*}
where $E$ is a finite set inside $\mathcal{V}$. The set $E$ and the functions $h_{\mathcal{C}},l_{s},P(x)$ are explicit and will be introduced in Section \ref{section 4}. The polytope $\Delta$ is the polytope associated to a section $s$ of $L$.

When $p = 1$, we get $\delta_{G}^{(1)}(X,L) = \delta_{G}(X,L)$, the $\delta-$invariant with $G-$action. 
\end{thm}
 
The key to prove the theorem is to compute the valuative invariant $S(L,v)$ for $v\in 
 \mathcal{V}$ (more generally $S^{(p)}(L,v)$), the expected vanishing order of $L$ along $v$. These invariants are introduced in Section \ref{section 2}.

Then we study $\mathrm{g}-$solitons on a $\mathbb{Q}-$Fano variety in the sense of \cite{berman2014complex}. In order to do this, we need to have a natural torus action at hand. For a spherical $\mathbb{Q}-$Fano $G-$variety $X$ which is an equivariant completion of the spherical homogeneous space $G/H$, $T = (N_{G}(H)/H)^{0}$ is a torus, and it is actually the identity component 
of $\mathrm{Aut}_{G}(X)$. Like above, there is a moment polytope $\Delta^{+}_{T}$ associated to the $T-$action. Fix now a smooth positive function $\mathrm{g}$ on $\Delta^{+}_{T}$. 

First we can compute the valuative invariant $S^{\mathrm{g}}(v)$, the $\mathrm{g}$-weighted expected vanishing order. Then we can get $\delta-$invariant specially designed for $\mathrm{g}-$solitons, namely $\delta_{G}^{\mathrm{g}}(X,-K_{X})$ in Section \ref{section_val_inv_g_soliton}.

\begin{thm}
Let $(X,-K_{X})$ be a $\mathbb{Q}-$Fano spherical $G$-variety which is an equivariant completion of $G/H$. Let $\Delta$ be the polytope associated to the $G-$action. Let $T = (N_{G}(H)/H)^{0}$ and we have the corresponding polytope $\Delta_{T}^{+}$. Assume $\mathrm{g}$ is a strictly positive smooth function on $\Delta_{T}^{+}$. Then
\begin{equation*}
\delta^{\mathrm{g}}_{G}(X,-K_{X}) = \min_{v\in E} \frac{h_{\mathcal{C}}(v)}{h_{\mathcal{C}}(v) + \langle \mathrm{bar}_{\mathrm{DH}}^{\mathrm{g}}(\Delta),v\rangle},
\end{equation*}    
where
\[\mathrm{bar}_{\mathrm{DH}}^{\mathrm{g}}(\Delta) = \frac{\int_{\Delta}\mathrm{g}(\bar{x})P(x)xdx}{\int_{\Delta}\mathrm{g}(\bar{x})P(x)dx}.\]
Here $E$ is a finite set inside $\mathcal{V}$. The set $E$ and the functions $h_{\mathcal{C}},P(x)$ are explicit and will be introduced in Section \ref{section 4}. The polytope $\Delta$ is the polytope associated to a preferable section $u$ of $-K_{X}$ introduced in Section \ref{section 4}. The $\bar{x}$ denotes the projection of $x\in \Delta$ on $\Delta_{T}^{+}$.
\end{thm}

Then we can directly get a result about the $G$-equivariantly $\mathrm{g}$-weighted Ding-semistability.

\begin{cor}
Let $(X,-K_{X})$ be a $\mathbb{Q}$-Fano spherical $G$-variety which is an equivariant completion of $G/H$. Let $\Delta$ be the polytope associated to the $G-$action. Let $T = (N_{G}(H)/H)^{0}$ and we have the corresponding polytope $\Delta_{T}^{+}$. Assume $\mathrm{g}$ is a smooth positive function on $\Delta_{T}^{+}$. Then the following statements are equivalent:
    \begin{enumerate}
        \item The $\mathrm{g}-$weighted barycenter $bar_{DH}^{\mathrm{g}}(\Delta)$ is in the dual cone of $-\mathcal{V}_{\mathbb{R}}$.
        \item The pair $(X,-K_{X})$ is G-equivariantly $\mathrm{g}$-weighted Ding-semistable.
    \end{enumerate}
\end{cor}

By using results from \cite{li2021notes,han2023yau} and the invariant $\beta^{\mathrm{g}}$, which is a generalization of the $\beta-$invariant defined in \cite{fujita2019valuative,Li2017k}, we have the following theorem on the existence of $\mathrm{g}-$solitons on a $\mathbb{Q}-$Fano spherical $G-$variety.

\begin{thm}\label{thm1.5}
Let $(X,-K_{X})$ be a $\mathbb{Q}$-Fano spherical $G$-variety which is an equivariant completion of $G/H$. Let $\Delta$ be the polytope associated to the $G-$action. Let $T = (N_{G}(H)/H)^{0}$ and we have the corresponding polytope $\Delta_{T}^{+}$. Assume $\mathrm{g}$ is a smooth positive function on $\Delta_{T}^{+}$. Then the following three statements are equivalent:
    \begin{enumerate}
        \item The $\mathrm{g}-$weighted barycenter $bar_{DH}^{\mathrm{g}}(\Delta)$ is in the relative interior of the dual cone of $-\mathcal{V}_{\mathbb{R}}$.
        \item The pair $(X,-K_{X})$ is G-equivariantly $\mathrm{g}$-weighted Ding-polystable for special test configurations.
        \item There exists a $\mathrm{g}-$soliton on $X$.
    \end{enumerate}
\end{thm}

After completing our article, we got informed of the recent work of Li-Li-Wang \cite{li2022weighted}, where they also prove our Theorem \ref{thm1.5} with a different method. Our approach is based on the valuative theory, and our computation of the $\beta-$invariant is elementary. Li-Li-Wang \cite{li2022weighted} tackles the problem using test configurations. Eventually, the two perspectives coincide on special test configurations. Note that a theory of $K-$stability for $\mathrm{g}-$solitons is also developed in \cite{li2022weighted} which has its own interest, whereas we only discuss Ding-stability based on \cite{li2021notes,han2023yau}.

With the above theorem and a special choice of the function $\mathrm{g}$, we obtain the following application.

\begin{thm}
    Let $X$ be a smooth Fano horospherical variety. Then the total space of the canonical bundle with the zero section removed,
\[
   K_X^\times := K_X \setminus \{0\},
\]
admits a Calabi-Yau cone metric, given by a Sasaki-Einstein structure on a unit circle inside $K_{X}$ with respect to some hermitian metric on $K_{X}$.
\end{thm}

\begin{ac}
    The author would like to express his gratitude to his thesis supervisors Julien Keller and Vestislav Apostolov for valuable suggestions. The author also benefits a lot from discussions with Thibaut Delcroix.
\end{ac}

\section{Invariants for filtrations}\label{section 2}

Let $X$ be a complex normal projective variety with $\mathbb{Q}-$Cartier $K_{X}$. We say that a prime divisor $E$ is over $X$ if there is a birational proper morphism $\mu:Y\rightarrow X$ with normal $Y$ such that $E$ is a prime divisor on $Y$. For such a $E$, the log discrepancy function is given by $A_{X}(E) = \mathrm{ord}_{E}(K_{Y/X}) + 1$, where $K_{Y/X}$ is the relative canonical divisor. We say $X$ is klt or has klt singularities if $A_{X}(E)>0$ for all prime divisors $E$ over $X$.

Let $X$ be a complex normal projective variety with klt singularities. Let $L$ be an ample line bundle on $X$. Let $d_{m} = \mathrm{dim}H^{0}(X,mL)$. Without loss of generality, up to a multiple of $L$, we can assume that $d_{m}\neq 0$ for any $m\in \mathbb{N}$. Let $R = \oplus_{m\geq 0}R_{m}$, where $R_{m} = H^{0}(X,mL)$.

By saying that $\mathcal{F}$ is a filtration of the ring $R$, we mean that for any $\lambda\in \mathbb{R}_{\geq 0}$ and $m\in \mathbb{N}$, there is a vector subspace $\mathcal{F}^{\lambda}R_{m}\subset R_{m}$ satisfying the properties:

\begin{enumerate}
    \item $\mathcal{F}^{\lambda}R_{m}\subset \F^{\lambda'}R_{m}$  for $\lambda\geq \lambda';$
    \item  $\mathcal{F}^{\lambda}R_{m} = \cap_{\lambda'<\lambda}\F^{\lambda'}R_{m};$
    \item $\F^{\lambda_{1}}R_{m_{1}}\cdot \F^{\lambda_{2}}R_{m_{2}} \subset \F^{\lambda_{1}+\lambda_{2}}R_{m_{1}+m_{2}};$
    \item $\F^{0}R_{m} = R_{m}$ and $\F^{\lambda}R_{m} = 0$ for $\lambda>>0.$
\end{enumerate}
The jumping numbers of the filtration $\F$ are defined as follows:
\[0\leq a_{m,1}(\F)\leq a_{m,2}(\F)\leq \cdots \leq a_{m,d_{m}}(\F),\]
where $a_{m,j}(\F) := \mathrm{inf}\{\lambda\in \mathbb{R}_{\geq 0}|\;\mathrm{codim}\F^{\lambda}R_{m}\geq j\}$.

By the second condition in the definition of filtrations, we see that $\F^{\lambda_{1}}R_{m} = \F^{\lambda_{2}}R_{m}$ for $\lambda_{1},\lambda_{2}\in (a_{m,i}(\F),a_{m,i+1}(\F)]$ when $a_{m,i}(\F)\neq a_{m,i+1}(\F)$. For $p\in [1,+\infty)$, the following invariant is defined in \cite[Section 2.2]{zhang2022valuative}:
\begin{equation}\label{def_Sp}
S^{(p)}_{m}(L,\F) := \frac{1}{d_{m}}\sum_{j=1}^{d_{m}}\left(\frac{a_{m,j}}{m}\right)^{p}.
\end{equation}
When $p=1$, this invariant was first introduced in \cite{blum2020thresholds}.

With a $\mathbb{R}$-valued valuation $v$ on $X$, we can define a filtration $\F_{v}$ by setting $\F_{v}^{\lambda}R_{m}:= \{s\in R_{m}|\; v(s)\geq \lambda\}$ for $\lambda \geq 0$ and $m\in \mathbb{N}$. In this case, we write $S^{(p)}_{m}(L,v)$ instead of $S^{(p)}_{m}(L,\F_{v})$. Especially, a prime divisor $F$ over $X$ provides a valuation $\mathrm{ord}_{F}$ which then leads to a filtration. In this case, as in \cite{zhang2022valuative}, we write $S_{m}^{(p)}(L,F)$ instead of $S_{m}^{(p)}(L,\F_{\mathrm{ord}_{F}})$. In this case, $S_{m}^{(p)}(L,F)$ is called the $p$-th moment of the expected vanishing order of $L$ along $F$ at level $m$. A little bit more generally, valuations in the form of $c\cdot \mathrm{ord}_{F}$ with a positive real number $c$ are called divisorial. They similarly provide a filtration. 

Similarly we define (see for example \cite[Section 2.4]{blum2020thresholds})
\begin{equation}\label{Tm}
T_{m}(L,\F):=\frac{a_{m,d_{m}}}{m}.
\end{equation}
The limit $T(L,\F) := \lim_{m\to +\infty}T_{m}(L,\F)$ always exists in $[0,+\infty]$ (see \cite[section 2.4]{blum2020thresholds}). We say that $\F$ is linearly bounded if $T(L,\F)<+\infty$. If $\F_{v}$ is linearly bounded for some valuation $v$, we say that such a valuation $v$ has linear growth.

We recall the following definition of \cite[Section 2.2]{zhang2022valuative}:
\begin{definition}
For a linearly bounded filtration $\F$, we define
\[S^{(p)}(L,\F) := \lim_{m\to +\infty}S_{m}^{(p)}(L,\F).\]
Especially, the above limit exists and is a finite number (proved for $p=1$ in \cite[Corollary 3.6]{blum2020thresholds}, shown for general $p$ in \cite[Proposition 2.5]{zhang2022valuative}). 
\end{definition}

Since $X$ is a complex normal projective variety with klt singularities, as explained in \cite{boucksom2015valuation}, we can define the log discrepancy function for general valuations $$A_{X}:\mathrm{Val}_{X}\rightarrow [0,+\infty]$$
where $\mathrm{Val}_{X}$ is the set of all $\mathbb{R}$-valued valuations on $X$. This function is lower semi-continuous and satisfies $A_{X}(tv) = tA_{X}(v)$ for $\lambda\in \mathbb{R}_{\geq 0}$. What is also important is that $A_{X}(v) = 0$ if and only if $v$ is the trivial valuation. The next lemma is shown in \cite[Lemma 3.1]{blum2020thresholds}.

\begin{lem}
Any $v\in \mathrm{Val}_{X}$ satisfying $A_{X}(v)<+\infty$ has linear growth. Especially, this is the case for a divisorial valuation.
\end{lem}

We define as in \cite{zhang2022valuative} the $\delta^{(p)}-$invariant:

\begin{definition}
For $p\in [1,+\infty)$, we set
\[\delta^{(p)}(L):= \inf_{v\in \mathrm{DivVal}_{X}^{*}}\frac{A_{X}(v)}{S^{(p)}(L,v)^{1/p}},\]
where the first $v$ runs through the set of non-trivial divisorial valuations over $X$, denoted by $\mathrm{DivVal}_{X}^{*}$. 
When $p=1,$ the important $\delta-$invariant was first introduced in \cite[Definition 0.2]{fujita2018k}.
\end{definition}

The classical $\alpha$-invariant looks very similar (see for example \cite[Corollary 4.2]{blum2020thresholds}):

\begin{definition}
We define
\[\alpha(L):= \inf_{v\in \mathrm{DivVal}_{X}^{*}}\frac{A_{X}(v)}{T(L,v)},\]
where the first $v$ runs through all non-trivial divisorial valuations over $X$.
\end{definition}

\begin{remark}
It is shown in \cite{blum2020thresholds,zhang2022valuative} that:
\begin{align*}
\delta^{(p)}(L)&= \inf_{v\in \mathrm{Val}_{X}^{*}}\frac{A_{X}(v)}{S^{(p)}(L,v)^{1/p}},\\
\alpha(L)&= \inf_{v\in \mathrm{Val}_{X}^{*}}\frac{A_{X}(v)}{T(L,v)},
\end{align*}
where $v$ runs through non-trivial valuations with finite log discrepancy.
\end{remark}

\section{Group action}\label{section 3}

As in \cite{golota2020delta}, we can take group actions into consideration. Let $G$ be a complex connected algebraic group. Assume that there is a $G$-action on $X$ and the ample line bundle $L$ is $G$-linearized. In \cite{golota2020delta}, only the case $G\subset \mathrm{Aut}(X,L)$ is considered, but extension to more general group actions does not introduce any essential difference. We can define:

\begin{definition}\label{definition_delta_p}
Let $X$ be a normal projective $G$-variety with klt singularities and $L$ be a $G$-linearized ample line bundle on $X$. For $p\in [1,+\infty)$, we define the $\delta^{(p)}-$invariant with $G-$action
\[\delta_{G}^{(p)}(L):= \inf_{v\in \mathrm{DivVal}_{X}^{G*}}\frac{A_{X}(v)}{S^{(p)}(L,v)^{1/p}}, \]
where $v$ runs through all non-trivial $G$-invariant divisorial valuations over $X$.
\end{definition}

The classical $\alpha_{G}$-invariant has been introduced in the literature (see for example \cite[Definition 2.20]{golota2020delta}):

\begin{definition}
Let $X$ be a normal projective $G$-variety with klt singularities and $L$ a $G$-linearized ample line bundle on $X$, the $\alpha-$invariant with $G-$action is
\[\alpha_{G}(L):= \inf_{v\in \mathrm{DivVal}_{X}^{G*}}\frac{A_{X}(v)}{T(L,v)},\]
where $v$ runs through all non-trivial $G$-invariant divisorial valuations over $X$. 
\end{definition}

\section{Preliminaries on spherical varieties}
Let $G$ be a connected reductive complex algebraic group, $B$ a Borel subgroup of $G$, and $T_{\max}$ a maximal torus of $B$.

\begin{definition}
A normal $G-$variety is spherical if it contains an open $B-$orbit. 
\end{definition}

Let $X$ be a spherical $G-$variety, then it must contain an open $G-$orbit. If we fix a point $x$ in this orbit, we can write $G\cdot x\cong G/H$, where $H$ is the closed subgroup of $G$ which fixes $x$. We call $G/H$ a spherical homogeneous space and $X$ a spherical embedding of $G/H$. We have the spherical lattice:
\[M = \{\chi \text{ a character of } B| b\cdot f_{\chi} = \chi(b)f_{\chi} \text{ for some }f_{\chi} \in \mathbb{C}(X)^{*} = \mathbb{C}(G/H)^{*} \text{ and all } b\in B\}\]
This is a sub-lattice of the character group of $B$. Note that $f_{\chi}$ is unique up to multiplication by $\mathbb{C}^{*}$. For convenience, we identify $\chi$ with $f_{\chi}$ in this paper whenever no confusion can arise. So we can view $\chi$ as a rational function on $G/H$. Let $N = \mathrm{Hom}_{\mathbb{Z}}(M,\mathbb{Z})$ be the dual lattice. 

For a valuation $v\in \mathrm{Val}_{X}$, we can define $\langle v, \chi \rangle = v(f_{\chi})$. In this way we see that there is a map $\rho:\mathrm{Val}_{X}\rightarrow N\otimes \mathbb{R}$. Let $\mathrm{QVal}^{G}_{X}$ be the set of $G$-invariant $\mathbb{Q}$-valued valuations on $X$ or equivalently on $G/H$. The map $\rho$ identifies the set $\mathrm{QVal}_{X}^{G}$ with a polyhedral convex cone $\mathcal{V}\subset N\otimes \mathbb{Q}$ (originally from \cite[Section 4.1]{brion1987valuations}, see also \cite[Section 3.1 and  Section 4.1]{brion1997varietes}). We always treat $\mathrm{QVal}_{X}^{G}$ as a subset of $N\otimes \mathbb{Q}$. We denote the closure of $\mathcal{V}$ inside $N\otimes \mathbb{R}$ as $\mathcal{V}_{\mathbb{R}}$.

Let $D(X)$ be the set of $B$-stable prime divisors in $X$. For a $G$-orbit $Y$ in $X$, we let $D_{Y}(X)\subset D(X)$ be the set of $B$-stable prime divisors which contain $Y$. Let $G_{Y}(X) =\{D\in D_{Y}(X)| D \text{ is }G-\text{invariant}\}$ and $\mathcal{F}_{Y}(X) =\{D\in D_{Y}(X)| D \text{ is not }G-\text{invariant}\}$. Then we let $\mathcal{C}_{Y}(X)$ be the cone inside $N\otimes \mathbb{Q}$ generated by $\rho(\mathcal{F}_{Y}(X))$ and $\rho(G_{Y}(X))$. The pair $(\mathcal{C}_{Y}(X),\mathcal{F}_{Y}(X))$ is called a colored cone. Let $\mathbf{F}(X)$ be the set of all the colored cones $(\mathcal{C}_{Y}(X),\mathcal{F}_{Y}(X))$, where $Y$ goes through all the $G-$orbits inside $X$, then $\mathbf{F}(X)$ is called a colored fan. 

\begin{remark}\label{rm3.1}
We can see that every element in $\mathrm{QVal}_{X}^{G}$ is divisorial. Actually for $v\in \mathcal{V}$, after sufficiently subdividing the cones of $\mathcal{C}_{X}$ and taking out colors, we can always find a toroidal $G$-equivariant resolution $Z$ of which $v$ is on an extremal ray of some colored cone in the corresponding colored fan. Then $v$ must be divisorial.
On the other hand, every $G$-invariant divisorial valuation is clearly proportional to some element in $\mathcal{V}$. As changing valuations by a positive constant doesn't modify ratios like $\frac{A_{X}(v)}{S^{(p)}(L,v)^{1/p}}$. We can just think of $\mathcal{V}$ as $\mathrm{DivVal}_{X}^{G}$. See more details in \cite[Remark 3.5]{pasquier2017survey} which is based on \cite{cox2000toric}. 
\end{remark}

We now turn to the combinatorial counterparts of the notions of colored cones and colored fans.

\begin{definition}
A \emph{colored cone} in $N \otimes \mathbb{Q}$ is a pair $(\mathcal{C},\mathcal{F})$, 
where $\mathcal{C} \subseteq N \otimes \mathbb{Q}$ and $\mathcal{F} \subseteq D(G/H)$, 
with $D(G/H)$ denoting the set of $B$-stable prime divisors of $G/H$, such that:
\begin{enumerate}
    \item the set $\mathcal{C}$ is a strictly convex polyhedral cone generated 
    by $\rho(\mathcal{F})$ and a finite number of elements in $\mathcal{V}$;
    \item the relative interior of $\mathcal{C}$ intersects $\mathcal{V}$;
    \item we have $0 \notin \rho(\mathcal{D})$.
\end{enumerate}
\end{definition}

\begin{definition}
A \emph{face} of a colored cone $(\mathcal{C},\mathcal{F})$ is a colored cone $(\mathcal{C}',\mathcal{F}')$ 
where $\mathcal{C}'$ is a face of $\mathcal{C}$ and 
\[
   \mathcal{F}' = \mathcal{F} \cap \rho^{-1}(\mathcal{C}').
\]
\end{definition}

\begin{definition}
A \emph{colored fan} in $N\otimes \mathbb{Q}$ is a collection $\mathbf{F}$ of colored cones such that:
\begin{enumerate}
    \item any face of a colored cone in $\mathbf{F}$ is still in $\mathbf{F}$;
    \item the relative interiors of the colored cones in $\mathbf{F}$ do not intersect.
\end{enumerate}
Given a spherical embedding $(X,x)$ of $G/H$, we have already defined its colored fan:
\[
   \mathbf{F}(X)  = \{(\mathcal{C}_{Y}(X),\mathcal{F}_{Y}(X))
   \text{ for any $G$-orbit $Y$ of $X$}\}.
\]
\end{definition}

The classification of spherical varieties is provided by the Luna–Vust theory \cite{luna1983plongements}.
\begin{thm}
The map $(X,x) \mapsto \mathbf{F}(X)$ induces a bijection between spherical
embeddings of $G/H$ and colored fans in $N\otimes \mathbb{Q}$.
\end{thm}

\section{Valuative invariants for linearized line bundles on a spherical variety}\label{section 4}
Note that $X$ is complete if and only if $\mathcal{V}\subset \bigcup_{(\mathcal{C}_{Y}(X),\mathcal{F}_{Y}(X))\in \mathbf{F}(X)}\mathcal{C}_{Y}(X)=: \mathcal{C}_{X}$ (see for example \cite[Corollary 12.14]{timashev2011homogeneous}). From now on, we always assume that $X$ is a projective spherical $G-$variety. 

In this section, we focus on a linearized line bundle on a spherical variety. 
We now start the procedure of computing the valuative invariants. The method here is highly influenced by \cite[Section 7]{blum2020thresholds} on toric line bundles.

Now let $L$ be an ample $G$-linearized line bundle on $X$. We assume that $s\in H^{0}(X,L)$ is a $B$-eigenvector with weight $\chi$. We can write:

\[div(s) = \sum_{D\in D(X)}n_{D}D.\]

This divisor provides a piecewise linear function $l_{s}$ (see \cite[Section 3]{brion1989groupe}) on $\mathcal{C}_{X}$ which is linear on each cone $\mathcal{C}_{Y}(X)$. We can explicitly describe the function as follows. For each $G$-orbit $Y$, We let $\chi_{s,Y}$ be an element in $M$ such that $\langle \rho(D),\chi_{s,Y}\rangle  = n_{D} $ for any $D\in D_{Y}(X)$. Then we let $l_{s}$ be the function given by $\chi_{s,Y}$ on $\mathcal{C}_{Y}(X)$. 

For $v\in \mathrm{QVal}_{X}^{G}$, it is important to know the value of $v(s)$. The following classic lemma is included here for readers' convenience.

\begin{lem}
For $v\in \mathrm{QVal}_{X}^{G}$, we have $v(s) = l_{s}(v)$.
\end{lem}

\begin{proof}
We know that $v$ is in the relative interior of a unique colored cone $(\mathcal{C}_{Y}(X),\mathcal{F}_{Y}(X))$. Then the center of $v$ is $\bar{Y}$ (\cite[Theorem 2.5]{knop1991luna}). We consider the set $X_{Y,B} = \{x\in X|Y\subset \overline{Bx}\}$. This is an open affine subset of $X$ (\cite[Proposition 2.2.1]{brion1997varietes}). We know that $\bar{Y}$ is a spherical variety (\cite[Corollary 2.3.1]{brion1997varietes}), so it must contain an open dense $B$-orbit. Especially, $Y\cap X_{Y,B}\neq \emptyset$. This shows that the valuation $v$ has a center on $X_{Y,B}$. From this it is also clear that $D_{Y}(X) = \{D\in D(X)|D\cap X_{Y,B} \neq \emptyset\}$. Now, we can write
\begin{align*}
    div(s)|_{X_{Y,B}} &= \sum_{D\in D_{Y}(X)}n_{D}D,\\
                      &= \sum_{D\in D_{Y}(X)}\langle \chi_{s,Y}, \rho(v_{D})D,\\
                      &= div(\chi_{s,Y})|_{X_{Y,B}}.
\end{align*}

Thus on $X_{Y,B}$, we have $s = \chi_{s,Y}e$, where $e$ is a local trivialization of $L$. Thus $v(s) = \langle \chi_{s,Y}, v\rangle$. This is the same as saying $v(s) = l_{s}(v)$.
\end{proof}

Next we recall some classical results on linearized line bundles over spherical varieties. The original references are \cite{brion1987image,brion1989groupe}.
A lot of the results are mentioned in \cite[Section 3]{delcroix2023uniform} and we shall use the same notations.

There is a canonical polytope related to an ample $G$-linearized line bundle which is called the moment polytope:

\begin{definition}\label{moment polytope}
Let $\mathcal{X}(B)$ be the character group of $B$. Let $\Delta^{+}_{k}$ be the subset of $\mathcal{X}(B)$ such that
\[H^{0}(X,kL)\cong \bigoplus_{\lambda\in \Delta^{+}_{k}} V_{\lambda},\]
where $V_{\lambda}$ corresponds to the simple $G$-module with highest weight $\lambda$. Note that $H^{0}(X,kL)$ is multiplicity-free (\cite[Theorem 2.1.1]{brion1997varietes}), meaning that each simple $G$-module can at most appear once in the decomposition.
Then $\bigcup_{r\in \mathbb{N}^{*}}\frac{\Delta^{+}_{k}}{k}$ is a rational polytope inside $\mathcal{X}(B)\otimes \mathbb{Q}$ (\cite[Proposition 1.2.3]{brion1997varietes}). We define the moment polytope $$\Delta^{+} = \overline{\bigcup_{r\in \mathbb{N}^{*}}\frac{\Delta^{+}_{k}}{k}},$$ where the closure is taken in $\mathcal{X}(B)\otimes \mathbb{R}$.
\end{definition}

The moment polytope is always contained in the closed positive Weyl chamber. There is another very useful polytope $\Delta$ in $M\otimes \mathbb{R}$ related to the section $s$:
\[\Delta = \{m\in M\otimes \mathbb{R}| \rho(D)(m) + n_{D}\geq 0 \text{ for all } D\in D(X)\}\]
We have the relation $\Delta^{+} = \Delta + \chi$ when we consider $M\otimes \mathbb{R}$ as a subspace of $\mathcal{X}(B)\otimes \mathbb{R}$.

We have an isomorphism of $G$-modules:

\[H^{0}(X,kL) \cong \bigoplus_{m\in M\cap k\Delta}V_{k\chi + m} .\]
Let $R^{+}$ be the positive root system of $(G,B,T_{max})$ and $\rho$ the half of the sum of positive roots. Then we have the Weyl formula
\begin{equation}\label{Weyl_formula}
dim(V_{\lambda}) = \prod_{\alpha \in R^{+}} \frac{\langle\lambda + \rho,\alpha\rangle}{\langle\rho,\alpha\rangle},
\end{equation}
where $\langle \beta, \alpha \rangle$ represents the pairing of $\beta$ with the coroot corresponding to $\alpha$.

Now we can talk about filtrations of $H^{0}(X,kL)$ induced by a valuation $v\in \mathrm{Val}_{X}^{G}$. We denote again $H^{0}(X,kL)$ as $R_{k}$. Because $v$ is $G$-invariant, the set $\F_{v}^{\lambda}R_{k} = \{u\in H^{0}(X,kL)|v(u)\geq \lambda\}$ is a $G$-module. Now $\F_{v}^{\lambda}R_{k}$ must be the direct sum of some simple $G$-module $V_{k\chi + m}$. 

Look at each $V_{k\chi + m}$ individually. It contains $s^{k}\chi^{m}$, where $\chi^{m}$ is a $B$-eigenfunction associated to the lattice point $m$. Since $V_{k\chi + m}$ is a simple $G$-module, we can conclude that all the non-zero elements in $V_{k\chi + m}$ have the same value when paired with $v$, and that is $kl_{s}(v) + \langle m,v\rangle$. 

Hence we have

\begin{equation}\label{filtration}
\mathcal{F}_{v}^{\lambda}R_{k} \cong \bigoplus_{\substack{m\in M\cap k\Delta\\ \langle m, v\rangle + kl_{s}(v)\geq \lambda}}V_{k\chi + m}.
\end{equation}

The following lemma follows from definitions \eqref{def_Sp} and \eqref{Tm}, the Weyl formula \eqref{Weyl_formula} and the filtration \eqref{filtration}.
\begin{lem}
For $v\in \mathrm{QVal}^{G}_{X}$, we have

\begin{align*}
    S^{(p)}_{k}(L,v) &=\frac{1}{\mathrm{dim}H^{0}(X,kL)}\sum_{m\in M\cap k\Delta}\mathrm{dim}V_{k\chi + m} \left(\frac{\langle m,v\rangle + kl_{s}(v)}{k}\right)^{p},\\
    &=\frac{\sum_{m\in M\cap k\Delta}\left(\frac{\langle m,v\rangle + kl_{s}(v)}{k}\right)^{p}\prod_{\alpha\in R^{+}}\frac{\langle k\chi+m+\rho,\alpha\rangle}{\langle \rho,\alpha\rangle}}{\sum_{m\in M\cap k\Delta}\prod_{\alpha\in R^{+}}\frac{\langle k\chi+m+\rho,\alpha\rangle}{\langle \rho,\alpha\rangle}};\\
    T_{k}(L,v) &= \frac{\max_{m\in M\cap k\Delta}(\langle m,v\rangle + kl_{s}(v))}{k}.
\end{align*}
\end{lem}

Taking limit as $k\rightarrow \infty$, we obtain the following result.

\begin{lem}
For $v\in \mathrm{QVal}^{G}_{X}$, we have:
\begin{align*}
S^{(p)}(L,v) &= \frac{\int_{\Delta}P(x)(\langle x,v\rangle+l_{s}(v))^{p}dx}{\int_{\Delta}P(x)dx},\\
T(L,v)  &= \max_{m\in \Delta}(\langle m,v\rangle + l_{s}(v)),
\end{align*}
where $P(x) = \prod_{\alpha\in R_{X}^{+}}\frac{\langle x+\chi,\alpha \rangle}{\langle \rho, \alpha \rangle}$ and $R_{X}^{+}$ is the set of positive roots not orthogonal to $\Delta^{+} = \chi +\Delta$. 
\end{lem}

Particularly for $p=1$, this yields

\begin{lem}
Let us denote $\frac{\int_{\Delta}P(x)xdx}{\int_{\Delta}P(x)dx}$ as $\mathrm{bar}_{\mathrm{DH}}(\Delta)$. For $v\in \mathrm{QVal}^{G}_{X}$, we have
\begin{align*}
S(L,v) &= \frac{\int_{\Delta}P(x)(\langle x,v\rangle+l_{s}(v))dx}{\int_{\Delta}P(x)dx}, \\
&= \langle \mathrm{bar}_{\mathrm{DH}}(\Delta),v\rangle + l_{s}(v),
\end{align*}
where $P(x) = \prod_{\alpha\in R_{X}^{+}}\frac{( x+\chi,\alpha )}{( \rho, \alpha )}$ and $R_{X}^{+}$ is the set of positive roots not orthogonal to $\Delta^{+} = \chi +\Delta$.  
\end{lem}

From now on, we assume that $K_{X}$ is a $\mathbb{Q}$-Cartier divisor. This actually implies that $X$ has klt singularities when $X$ is spherical (see \cite[Proposition 5.6]{pasquier2017survey}).

From \cite[Section 4]{gagliardi2015gorenstein}, which is based upon \cite{brion1997curves} (see also \cite[Section 3.2]{delcroix2020k}), there is a section $u$ of $-K_{X}$ which is a $B$-eigenvector with weight $2\rho_{P} = \sum_{\alpha\in \Phi_{R_{u}(P)}}\alpha$. Here $P$ is the stabilizer of the open $B-$orbit, $R_{u}(P)$ is the unipotent radical of $P$, and then $\Phi_{R_{u}(P)}$ is the set of roots of $R_{u}(P)$. We have

\[div(u) = \sum_{D\text{ is }G-\text{stable}}D + \sum_{D\in D(G/H)}a_{D}D,\]
where the notation $D(G/H)$ refers to the set of $B$-stable prime divisors of $G/H$.

\begin{remark}
More precisely, after replacing $u$ by a suitable tensor power $u^{\otimes k}$, 
it becomes a section of some line bundle. 
Equivalently, for some positive integer $k$, the divisor $k\mathrm{div}(u)$ is Cartier.
\end{remark}

Similar to the construction of the piecewise-linear function $l_{s}$ related to a section $s$, the assumption that $K_{X}$ is a $\mathbb{Q}$-Cartier divisor is equivalent to the assumption that there is a piecewise linear function $h_{\mathcal{C}} := l_{u}$ on $\mathcal{C}_{X} = \bigcup_{(\mathcal{C}_{Y}(X),\mathcal{F}_{Y}(X))\in \mathbf{F}(X)}\mathcal{C}_{Y}(X)$, linear on each cone $\mathcal{C}_{Y}(X)$ given by some $m\in M\otimes \mathbb{Q}$. It is a very important fact that this function $h_{\mathcal{C}}$ is exactly the log discrepancy function $A_{X}$ when restricted on $\mathcal{V}$ (see \cite[Section 5]{pasquier2017survey}). Note again that $\mathcal{V}$ is a subset of $\mathcal{C}_{X}$.

The polyhedral cone $\mathcal{V}$ is carved into finitely many smaller polyhedral cones $\mathcal{C}_{Y}(X)\cap \mathcal{V}$. Let's gather the primitive generators of all the extremal rays of all the cones $\mathcal{C}_{Y}(X)\cap \mathcal{V}$ in a finite set $E$. Then we have:
\begin{thm}\label{thm 3.1}
Let $G$ be a connected reductive algebraic group and $X$ a spherical $G-$variety with $\mathbb{Q}$-Cartier $K_{X}$. Let $L$ be a $G$-linearized ample line bundle on $X$. We have
\begin{align*}
 \delta^{(p)}_{G}(X,L) &= \min_{v\in E} \frac{h_{\mathcal{C}}(v)}{\left(\frac{\int_{\Delta}P(x)(\langle x,v\rangle +l_{s}(v))^{p}dx}{\int_{\Delta}P(x)dx}\right)^{1/p}},\\
\alpha_{G}(X,L) &= \min_{v\in E} \min_{m\in \Delta}\frac{h_{\mathcal{C}}(v)}{\langle m,v\rangle + l_{s}(v)}.
\end{align*}
\end{thm}

\begin{proof}
From Definition \ref{definition_delta_p} and Remark \ref{rm3.1}, we have:

\[\delta^{(p)}_{G}(L) = \inf_{v\in \mathcal{V}\backslash \{0\}}\frac{A_{X}(v)}{S^{(p)}(L,v)^{1/p}}\]

The polyhedral cone $\mathcal{V}$ is carved into finitely many smaller polyhedral cones $\mathcal{C}_{Y}(X)\cap \mathcal{V}$. Let's say $\mathcal{C}_{Y}(X)\cap \mathcal{V}$ is generated by $\{v_{1},\cdots,v_{n}\}$. Then every element $v$ in the cone can be written as $v = a_{1}v_{i_1} + \cdots a_{k}v_{i_k}$, where every $a_{i}>0$. Note that $l_{s}$ is linear on each $\mathcal{C}_{Y}(X)$. Then we see

\begin{align*}
    S^{(p)}(L,v)^{1/p} &= \left(\frac{\int_{\Delta}P(x)(\langle x,v\rangle+l_{s}(v))^{p}dx}{\int_{\Delta}P(x)dx}\right)^{1/p},\\
    &= \left(\frac{\int_{\Delta}P(x)(\sum_{j=1}^{k} a_{j}(\langle x,v_{i_{j}}\rangle +l_{s}(v_{i_{j}})))^{p}dx}{\int_{\Delta}P(x)dx}\right)^{1/p},\\
    &\leq  \sum_{j=1}^{k} a_{j} \left(\frac{\int_{\Delta}P(x)(\langle x,v_{i_{j}}\rangle +l_{s}(v_{i_{j}}))^{p}dx}{\int_{\Delta}P(x)dx}\right)^{1/p},\\
    & =  \sum_{j=1}^{k} a_{j} S^{(p)}(L,v_{i_{j}})^{1/p}.
\end{align*}
Thus we see that

\[\frac{A_{X}(v)}{S^{(p)}(L,v)^{1/p}} \geq \min_{v_{i_{j}}}\frac{h_{\mathcal{C}}(v_{i_{j}})}{S^{(p)}(L,v_{i_{j}})^{1/p}},\]
since $h_{\mathrm{C}}$ is linear on the $\mathcal{C}_{Y}\cap \mathcal{V}$. Similarly:

\[\frac{A_{X}(v)}{T(L,v)} \geq \min_{v_{i_{j}}}\frac{h_{\mathrm{C}}(v_{i_{j}})}{T(L,v_{i_{j}})}.\]
\end{proof}
Especially, when $p$ is 1, we get

\begin{cor}
With the same setup as above, we have
\[\delta_{G}(X,L) = \min_{v\in E} \frac{h_{\mathcal{C}}(v)}{\langle \mathrm{bar}_{\mathrm{DH}}(\Delta),v\rangle + l_{s}(v)}.\]  
\end{cor}

\section{Examples}

Several results in the literature can be recovered as consequences of what we proved previously.

\begin{exm}
For the Fano case and $L = -K_{X}$, as we mentioned above, we can take $s = u$. Then $l_{s}$ coincides with $h_{\mathcal{C}}$ and $\chi = 2\rho_{P}$. As for $\mathrm{bar}_{{DH}}(\Delta)$, we have

\begin{align*}
\mathrm{bar}_{{DH}}(\Delta) &= \frac{\int_{\Delta}\prod_{\alpha\in R_{X}^{+}}\frac{( x+\chi,\alpha )}{( \rho, \alpha )}xdx}{\int_{\Delta}\prod_{\alpha\in R_{X}^{+}}\frac{( x+\chi,\alpha )}{( \rho, \alpha )}dx},\\
& = \frac{\int_{\Delta+\chi}\prod_{\alpha\in R_{X}^{+}}\frac{( x ,\alpha )}{( \rho, \alpha )}(x-\chi)dx}{\int_{\Delta+\chi}\prod_{\alpha\in R_{X}^{+}}\frac{( x,\alpha )}{( \rho, \alpha )}dx},\\
&=  \frac{\int_{\Delta^{+}}\prod_{\alpha\in R_{X}^{+}}\frac{( x ,\alpha )}{( \rho, \alpha )}(x-2\rho_{P})dx}{\int_{\Delta^{+}}\prod_{\alpha\in R_{X}^{+}}\frac{( x,\alpha )}{( \rho, \alpha )}dx},\\
&= \mathrm{bar}_{\mathrm{DH}}(\Delta^{+}) - 2\rho_{P}.
\end{align*}

Thus in the Fano case, we recover \cite[Propostion1.4]{golota2020delta}:

\[\delta_{G}(X,-K_{X}) = \min_{v\in E}\frac{h_{\mathcal{C}}(v)}{h_{\mathcal{C}}(v) + \langle \mathrm{bar}_{\mathrm{DH}}(\Delta^{+})-2\rho_{P},v\rangle}.\]
\end{exm}

\begin{remark}
In \cite[Proposition 5.4]{golota2020delta}, there is a constant $V$ in front of the term $\langle \mathrm{bar}_{\mathrm{DH}}(\Delta^{+})-2\rho_{P},v\rangle$. From our considerations, this constant is actually 1. 
\end{remark}

\begin{exm}
In \cite{delcroix2015log}, Delcroix studies group compactifications and compute $\alpha-$invariant with group actions for ample line bundles. His formula for the anticanonical line bundle of the unique wonderful compactification of a semisimple adjoint group $\hat{G}$ (which is Fano) is beautifully simple and illustrative

\[\alpha_{\hat{G}\times \hat{G}}(X,-K_{X}) = \max\{c|2\rho \in (1-c)\Delta^{+}\}.\]

We reinterpret the formula within our framework. Let $\hat{G}$ be a semisimple adjoint group. We assume the Borel subgroup $\hat{B}$ and the maximal torus $\hat{T}_{max}\subset \hat{B}$ of $\hat{G}$ are fixed. Thus we have also fixed positive roots, simple roots and the positive Weyl chamber $C^{+}$ of $\hat{G}$. The variety $\hat{G}$ is considered as a $G = \hat{G}\times\hat{G}$-variety by the action $(g,h)\cdot x = gxh^{-1}$. The corresponding $H$ is the diagonal of $\hat{G}\times \hat{G}$, so $\hat{G}\times\hat{G}/H\cong \hat{G}$ by identifying $\overline{(g,h)}$ and $gh^{-1}$. The corresponding Borel subgroup of $G$ is $\hat{B}\times \hat{B}^{-}$. The corresponding maximal torus of $G$ is $\hat{T}_{max}\times \hat{T}_{max}$. The spherical lattice of $G/H$ can be identified with the character group of $\hat{B}$, denoted as $\mathcal{X}(\hat{B})$, by sending a character of $\hat{B}\times \hat{B}^{-}$ to the character of $\hat{B}$. With this identification, several objects, including the  algebraic moment polytope of $-K_{X}$, can be identified with objects defined with $\hat{B}$. Note that $\mathcal{X}(\hat{B}) = \mathcal{X}(\hat{T}_{\mathrm{max}})$, whereas the latter is isomorphic to the root lattice since $\hat{G}$ is a semisimple adjoint group. 

Let $\{\alpha_{i}|1\leq i\leq r\}$ be the set of simple roots, where $r$ is the dimension of the torus $\hat{T}_{max}$. Let $Q$ be the polytope constructed as the convex hull of the images of $2\rho + \sum_{i=1}^{r}\alpha_{i}$ under the action of the Weyl group. Then we have $\Delta^{+} = Q\cap C^{+}$. We refer to \cite[Proposition 6.1.11]{brion2007frobenius}.

Let's look at $\Delta^{+}$ in this case more carefully. To each positive root $\beta$, there is a corresponding reflection $s_{\beta}$. This $s_{\beta}$ moves any $x\in C^{+} $ in the $-\beta$ direction. Since simple roots generate positive roots in a nonnegative way, we see that the extremal rays coming from $2\rho + \sum_{i=1}^{r}\alpha_{i}$ are given by $-\alpha_{i}$. Thus the codimensional 1 faces containing $2\rho + \sum_{i=1}^{r}\alpha_{i}$  are generated by $\{\alpha_{i_{1}},\cdots, \alpha_{{i_{r-1}}}\}$. In other words, each of the face is perpendicular to some $v_{i}$, where $\{v_{i}|1\leq i\leq r\}$ is the dual basis of $\alpha_{i}$ in $N\otimes \mathbb{R}$. Because of all the considerations,
\begin{align*}
\Delta^{+} &= \{x\in C^{+}| \langle 2\rho + \sum_{i=1}^{r}\alpha_{i}, v_{i}\rangle \geq \langle x,v_{i}\rangle \text{ for all }v_{i} \},\\
&= \{x\in C^{+}| \langle 2\rho, v_{i}\rangle + 1 \geq \langle x,v_{i}\rangle \text{ for all }v_{i} \}.
\end{align*}
From this we get:

\begin{align*}
\alpha_{G}(X,-K_{X}) &= \max\{c|2\rho \in (1-c)\Delta^{+}\},\\
&= 1-\min\{\lambda|\lambda(\langle 2\rho,v_{i}\rangle + 1)\geq \langle 2\rho,v_{i}\rangle \text{ for all }v_{i}\},\\
&= 1-\max_{i}\left\{\frac{\langle 2\rho,v_{i}\rangle}{\langle 2\rho,v_{i}\rangle + 1}\right\},\\
&= \min_{i}\left\{\frac{1}{1 + \langle 2\rho,v_{i}\rangle}\right\}.
\end{align*}

Let's link this to our formula in Theorem \ref{thm 3.1}. The wonderful compactification $X$ enjoys several good properties. First, it is symmetric. This implies that the valuation cone $\mathcal{V}$ is exactly the cone generated by $\{-v_{i}|1\leq i\leq r\}$ (See \cite[Section 26]{timashev2011homogeneous} or \cite[Section 5.4.2]{delcroix2020k}). Second, it is the canonical embedding of $\hat{G}$ (see [Proposition 3.3.1]\cite{pezzini2010lectures_wonderful}). This implies the set $\{-v_{i}\}$ corresponds exactly to the set of $G-$invariant prime divisors on $X$. What's more, the colored fan of $X$ consists of a single colored cone $(\mathcal{V},\emptyset).$

In this case, the section $u$ mentioned above has weight $2\rho_{P} = 2\rho$. The function $ h_{\mathrm{C}}$ is 1 on each $-v_{i}$ and extends linearly. Now we use Theorem \ref{thm 3.1}:
\begin{align*}
    \alpha_{G}(X,-K_{X}) &= \min_{i}\min_{m\in \Delta}\left\{\frac{1}{\langle m,-v_{i}\rangle + 1}\right\},\\
    & = \min_{i}\min_{m\in \Delta^{+}}\left\{\frac{1}{\langle m-2\rho,-v_{i}\rangle + 1}\right\},\\
    & = \min_{i}\left\{\frac{1}{\langle 2\rho,v_{i}\rangle + 1}\right\},
\end{align*}
since $\langle m, -v_{i}\rangle \leq 0$ for $m\in \Delta^{+}\subset C^{+}$.
\end{exm}

\begin{remark}
Let's discuss $P(x)$ in the computation of $\delta_{G}^{(p)}(-K_{X})$ in the previous example. Let's denote the set of positive roots of $\hat{G}$ with respect to $\hat{B}$ as $\Phi^{+}$. After identifying the spherical lattice of $G/H$ with $\mathcal{X}(\hat{B})$, we should write $P(x) = \prod_{\alpha\in \Phi^{+}}(\frac{\langle x+2\rho,\alpha \rangle}{\langle \rho, \alpha \rangle})^{2}$ instead of $\prod_{\alpha\in \Phi^{+}}\frac{\langle x+2\rho,\alpha \rangle}{\langle \rho, \alpha \rangle}$, where $\rho$ is the half sum of positive roots of $\hat{G}$ with respect to $\hat{B}$. The square here really corresponds to the fact that we need to consider representations of $G$ instead of $\hat{G}$ in the process of the last section.
\end{remark}

Now let's look at a last and very concrete example.

\begin{exm}

    Let $G$ be $\mathrm{PGL}_{2}(\mathbb{C})\times \mathrm{PGL}_{2}(\mathbb{C})$. Let $H$ be the diagonal of $\mathrm{PGL}_{2}(\mathbb{C})\times \mathrm{PGL}_{2}(\mathbb{C})$. Then $G/H\cong \mathrm{PGL}_{2}(\mathbb{C})$ identifying $\overline{(a,b)} $ with $ab^{-1}$. Then the induced group action on $\mathrm{PGL}_{2}(\mathbb{C})$ is $(g,h)\cdot x = gxh^{-1}$ for $g,h,x\in \mathrm{PGL}_{2}(\mathbb{C})$. A usual choice of a Borel subgroup is $B = \{(b_{1},b_{2})\in G|b_{1}\text{ upper-triangular, } b_{2} \text{ lower-triangular}\}$. Let $T_{max} = \{(b_{1},b_{2})\in G|b_{1} \text{ and }b_{2} \text{ are diagonal}\}$. Define:
    \[f\left(\begin{bmatrix} a&b\\c&d \end{bmatrix}\right) = \frac{d^2}{ad-bc}.\]

    It is easy to check that $f$ is a $B$-eigenvector, actually:

    \[\left(\left(\begin{bmatrix} \lambda_{1}&\lambda_{2}\\0&\lambda_{3}\end{bmatrix},\begin{bmatrix} \beta_{1}&0\\\beta_{2}&\beta_{3} \end{bmatrix}\right)\cdot f\right)\left(\begin{bmatrix} a&b\\c&d \end{bmatrix}\right) = \frac{\lambda_{1}}{\lambda_{3}}\frac{\beta_{3}}{\beta_{1}}\frac{d^{2}}{ad-bc}.\]
   
    In this case, the spherical lattice of $G/H$ is isomorphic to $\mathbb{Z}$ and is generated by the element $\omega$ given by $\omega\left(\begin{bmatrix} \lambda_{1}&\lambda_{2}\\0&\lambda_{3}\end{bmatrix},\begin{bmatrix} \beta_{1}&0\\\beta_{2}&\beta_{3} \end{bmatrix}\right) = \frac{\lambda_{1}}{\lambda_{3}}\frac{\beta_{3}}{\beta_{1}}$. We can identify $\omega$ with $\omega_{1}\left(\begin{bmatrix} \lambda_{1}&\lambda_{2}\\0&\lambda_{3}\end{bmatrix}\right) = \frac{\lambda_{1}}{\lambda_{3}}$. The character $\omega_{1}$ generates the character group of the Borel subgroup of the first component of $G$. We identify $\omega$ or $\omega_{1}$ with 1. 

    The wonderful compactification of $\mathrm{PGL}_{2}(\mathbb{C})$ is $X = \mathbb{P}(M_{2\times 2}(\mathbb{C}))$. The only closed $G$-orbit of $X$ is the prime divisor $D = \biggl\{\begin{bmatrix} a&b\\c&d \end{bmatrix}|ad-bc = 0\biggr\}$. The embedding $X$ is simple, meaning that there is only one closed $G-$orbit, and complete. It is colorless simply because the only closed $G$-orbit is of codimension 1. Clearly $v_{D}(f) = -1$. The valuation cone is generated by $\rho(v_{D})$ in a nonnegative way. This cone can be easily identified with the dual of the negative Weyl chamber of the first component of $G$. 

    Now let's look at the section $u$ of the anticanonical bundle of $X$. Clearly $-K_{X}\cong \mathcal{O}(4)$. We claim that the section $u$ is identified with the section $(ad-bc)d^{2}$ of $\mathcal{O}(4)$. We now verify this fact.

    Actually, the section $(ad-bc)d^{2}$, when written locally in the chart $\{d \neq 0\}$ and considered as a section of $-K_{X}$, is $(a-bc)\frac{\partial}{\partial a}\wedge\frac{\partial}{\partial b}\wedge\frac{\partial}{\partial c}$. We have:

    \[\begin{bmatrix} \lambda_{1}&\lambda_{2}\\0&\lambda_{3} \end{bmatrix}\begin{bmatrix} a&b\\c&1 \end{bmatrix} = \begin{bmatrix} \lambda_{1}a+\lambda_{2}c&\lambda_{1}b+\lambda_{2}\\ \lambda_{3}c&\lambda_{3} \end{bmatrix} = \begin{bmatrix} \frac{\lambda_{1}a+\lambda_{2}c}{\lambda_{3}}&\frac{\lambda_{1}b+\lambda_{2}}{\lambda_{3}}\\ c&1 \end{bmatrix}.\]

    Let $a' = \frac{\lambda_{1}a+\lambda_{2}c}{\lambda_{3}},b' = \frac{\lambda_{1}b+\lambda_{2}}{\lambda_{3}},c' = c$. Then $a=\frac{\lambda_{3}a'-\lambda_{2}c'}{\lambda_{1}},b = \frac{\lambda_{3}b'-\lambda_{2}}{\lambda_{1}},c=c'$. Then:

    \[\begin{bmatrix} \lambda_{1}&\lambda_{2}\\0&\lambda_{3} \end{bmatrix}_{*}\left((a-bc)\frac{\partial}{\partial a}\wedge\frac{\partial}{\partial b}\wedge\frac{\partial}{\partial c}\right) = \frac{\lambda_{1}}{\lambda_{3}}\left((a'-b'c')\frac{\partial}{\partial a'}\wedge\frac{\partial}{\partial b'}\wedge\frac{\partial}{\partial c'}\right).\]
    Similar computation holds if we let $\begin{bmatrix} \beta_{1}&0\\\beta_{2}&\beta_{3} \end{bmatrix}$ act on the right side of the section. This shows that $(ad-bc)d^{2}$, when considered as a section of $-K_{X}$, is a $B$-eigenvector with weight $\omega$. So it must be the preferable section $u$.

    It is not hard to see that the moment polytope $\Delta^{+}$ is $[0,2]$. Actually there are only two $B$-stable divisors, namely $D$ and the divisor given by $\{d=0\}$. Then we can compute the polytope associated to the section $u$ and then get the moment polytope. Otherwise we can also conclude that by considering dimensions. We know that $H^{0}(X,\mathcal{O}(4))$ is of dimension 35. On the other hand, $\mathrm{dim}(V_{i}) = (2i+1)^2$ for $i=0,1,2$. At the same time, we see that the polytope $\Delta$ associated to $u$ is $[-1,1]$.
    
    Now using the formula from Theorem \ref{thm 3.1}:
    \[\alpha_{G}(X,-K_{X}) = \min_{m\in [-1,1]}\left\{\frac{1}{\langle m, -1\rangle + 1} \right\} = \frac{1}{2}.\]
    This was also shown in \cite{delcroix2015log}.
    For the $\delta_{G}^{(p)}$-invariant, We obtain

    \[\delta_{G}^{(p)}(X,-K_{X}) = \frac{1}{\left(\frac{\int_{-1}^{1}\langle x+1,1\rangle^{2}(\langle x,-1\rangle +1)^{p}dx}{\int_{-1}^{1}\langle x+1,1\rangle^{2} dx}\right)^{1/p}} = \frac{1}{2}\left(\frac{(p+1)(p+2)(p+3))}{6}\right)^{1/p}.\] 
\end{exm}

\section{Valuative invariants for g-solitons}\label{section_val_inv_g_soliton}

A $\mathbb{Q}-$Fano variety is a normal projective $\mathbb{Q}$-Gorenstein Fano variety with klt singularities. Assume that $X$ is a $\mathbb{Q}-$Fano variety of dimension $m$. Assume also that there is a torus action on $X$. We denote the algebraic torus as $T$ and let $T^{c}$ be the compact real torus inside $T$. We introduce $\mathrm{g}-$solitons as in \cite{han2023yau}.

Let $h$ be a smooth positive hermitian metric on $-K_{X}$ with positive curvature $\omega$ representing $2\pi c_{1}(X)$. We can define a global measure on $X$ as (see \cite{berman2019kahler})
\[d\mu_{0} = |s|_{h}^{\frac{2}{r}}((\sqrt{-1})^{rm^{2}}s^{*}\wedge \overline{s^{*}})^{\frac{1}{r}},\]
where $r$ is sufficiently divisible so that $-rK_{X}$ is a line bundle, $s$ is a nowhere-zero local section of $-rK_{X}$, $s^{*}$ is the dual of $s$. 

For any Kähler form $\omega_{\varphi} = \omega + \sqrt{-1}\partial\bar{\partial}\varphi$ inside $2\pi c_{1}(X)$, there is an associated moment map:

\[m_{\varphi}:X\rightarrow \mathrm{Lie}(T^{c})\cong M_{T}\otimes \mathbb{R},\]
where $\mathrm{Lie}(T^{c})$ is the lie algebra of $T^{c}$ and $M_{T}$ is the character lattice of $T$.

The image of $m_{\varphi}$ is the symplectic moment polytope $-\Delta_{T}^{+}$ (we will introduce the notation $\Delta_{T}^{+}$ below), which is independent of the choice of $\omega_{\varphi}$. Let $\mathrm{g}$ be a strictly positive smooth function on $-\Delta_{T}^{+}$. The $\mathrm{g}-$soliton equation is the following non-linear PDE in $\varphi$ (we refer to \cite{berman2014complex,han2023yau,li2021notes}):

\[\mathrm{g}\circ m_{\varphi}\frac{\omega_{\varphi}^{n}}{n!} = e^{-\varphi}d\mu_{0}. \]

In \cite{rubinstein2021basis}, the authors introduce $\delta$-invariants for $\mathrm{g}-$solitons. We recall some of their definitions. The setup is as above. From the torus action, we get an induced $T-$linearization of $-K_{X}$, thus we have a decomposition $H^{0}(X,-rK_{X}) = \bigoplus_{\alpha} R_{r,\alpha}$, here $\alpha$ is a character of the torus and $R_{r,\alpha}$ represents the subspace of eigenvalue $\alpha$. Denote $\Delta^{+}_{T,r}$
as the set of $\alpha$ such that $R_{r,\alpha}$ is non-trivial. We denote the character lattice of $T$ as $M_{T}$. Then let $$\Delta^{+}_{T} = \overline{\cup_{r\in \mathbb{N}^{*}}\frac{\Delta_{T,r}}{r}}$$ which is a polytope in $M_{T}\otimes \mathbb{R}$. By \cite[Section 2.2]{brion1987image}, this polytope is the opposite of the symplectic moment polytope introduced above, justifying the notations. The strictly positive smooth function $\mathrm{g}$ on $-\Delta_{T}^{+}$ naturally induces a strictly positive smooth function on $\Delta_{T}^{+}$ which we still denote as $\mathrm{g}.$

Let $d_{r}$ be the dimension of $H^{0}(X,-rK_{X})$ and $d_{r,\alpha}$ be the dimension of $R_{r,\alpha}$. Let $\bar{\mathrm{g}}_{r} = \frac{1}{d_{r}}\sum_{\alpha\in\Delta_{T,r}^{+}}\mathrm{g}(\frac{\alpha}{r})d_{r,\alpha}$. Let $\mathcal{F}$ be a linearly bounded $T$-invariant filtration of $R = \bigoplus_{r\geq 0}H^{0}(X,-rK_{X})$ on $X$, meaning that every $\mathcal{F}^{\lambda}R_{m}$ is $T$-invariant. Let $a_{r,\alpha,j}$ be the j-th jumping number of the filtration on $R_{r,\alpha}$. Then

\begin{align*}
S_{r}^{\mathrm{g}}(\mathcal{F}) &:= \frac{1}{rd_{r}\bar{\mathrm{g}}_{r}} \sum_{\alpha\in \Delta_{T,r}^{+}}\sum_{j\geq 1}\mathrm{g}\left(\frac{\alpha}{r}\right)a_{r,\alpha,j},\\
S^{\mathrm{g}}(\mathcal{F}) &:= \lim_{r\rightarrow \infty} S_{r}^{\mathrm{g}}(\mathcal{F}).
\end{align*}
Now we define $$\delta^{\mathrm{g}}(X,-K_{X}) := \inf_{v\in \mathrm{DivVal}^{T*}_{X}}\frac{A_{X}(v)}{S^{\mathrm{g}}(v)}$$.

At the moment, $X$ is a $T$-variety. Now we assume that it is also a $G$-variety for some connected reductive $G$, and the $G$-action commutes with the $T$-action. In this way, $X$ is a $T\times G$-variety. 

\begin{definition}
    Let $X$ be a $\mathbb{Q}-$Fano variety. Assume that there is a $T\times G$-action on $X$, where $G$ is a connected complex reductive algebraic group, and $T$ is an algebraic torus. Let $\mathrm{g}$ be a strictly positive smooth function on $\Delta^{+}_{T}$. Then we introduce
\[\delta_{G}^{\mathrm{g}}(X,-K_{X}):= \inf_{v\in \mathrm{DivVal}_{X}^{T\times G*}}\frac{A_{X}(v)}{S^{\mathrm{g}}(v)},\]
where $\mathrm{DivVal}_{X}^{T\times G*}$ denotes the set of non-trivial $T\times G$-invariant divisorial valuations over $X$.
\end{definition}

We should point out that, in the Fano case, $\delta$-invariant is closely related to $\beta$-invariant. We recall the definition of the $\beta$-invariant here.

\begin{definition}
    Let $X$ be a $\mathbb{Q}-$Fano variety. Assume that there is a $T$-action on $X$, where $T$ is a algebraic torus. Let $\mathrm{g}$ be a strictly positive smooth function on $\Delta^{+}_{T}$. Then
    \[\beta^{\mathrm{g}}(v) := A_{X}(v) - S^{\mathrm{g}}(v),\]
    where $v$ is a $T$-invariant divisorial valuation.
\end{definition}

When we study $\mathrm{g}-$solitons on a spherical variety, we need to have a natural torus action. We know that the $G$-equivariant automorphism group of a spherical homogeneous space $G/H$ is $N_{G}(H)/H$, where $N_{G}(H)$ is the normalizer of $H$ inside $G$. The group action of $N_{G}(H)/H$ on $G/H$ is $pH\cdot gH = gp^{-1}H$. The group $N_{G}(H)/H$ is diagonalizable (\cite[Theorem 4.3]{brion1997varietes}). We elaborate on this fact a little bit more. For a $B-$eigenfunction $f_{\chi}\in \mathbb{C}(G/H)$ with eigenvalue $\chi$, because the $N_{G}(H)/H$-action commutes with the $G$-action, we see that $\gamma\cdot f_{\chi}$ is another $B-$eigenvector with eigenvalue $\chi$. Since $\mathbb{C}(G/H)$ is multiplicity-free, thus there is a non-zero complex number $\theta_{\chi}(\gamma)$ such that $\gamma\cdot f_{\chi} = \theta_{\chi}(\gamma)f_{\chi}$. In this way, there is a group homomorphism:

\begin{equation}\label{Theta_embedding}
\theta:N_{G}(H)/H\rightarrow \mathrm{Hom}(M,\mathbb{C}^{*}),
\end{equation}
where $M$ is the spherical lattice. It turns out this homomorphism is injective and the image is $\mathrm{Hom}(M/\langle \Sigma \rangle,\mathbb{C}^{*})$. See \cite[Section 4.2]{brion1997varietes} for the explication of $\Sigma$.

The neutral component of $N_{G}(H)/H$ is an algebraic torus, and it is also the neutral component of the group of $G-$equivariant automorphisms of $X$. This follows from the uniqueness result of \cite{losev2009uniqueness} and the Luna-Vust theory \cite{luna1983plongements}. See \cite[Section 3.1.3]{delcroix2020k} for more explications.  The neutral component $(N_{G}(H)/H)^{0}$ is the algebraic torus that we are going to use. Notice that $(N_{G}(H)/H)^{0}$ can be identified with $\mathrm{Hom}(M/\langle \Sigma' \rangle,\mathbb{C}^{*})$, where $M/\langle \Sigma' \rangle$ is $M/\langle \Sigma \rangle$ quotient by its torsion part.

The situation here is simpler than general cases thanks to the following lemma.

\begin{lem}
    Assume $v\in \mathrm{QVal}_{X}^{G}$. Then $v$ is also $N_{G}(H)/H$-invariant.
\end{lem}

\begin{proof}
The two groups $G$ and $N_{G}(H)/H$ commute with each other when they act on $G/H$, then they commute when they act on the space of rational functions on $G/H$, then it follows that they commute when they act on the space of valuations on $G/H$. Then we see easily that, if $v$ is $G$-invariant, $\gamma\cdot v$ is also $G$-invariant for $\gamma\in N_{G}(H)/H$. We just need to show that $\gamma\cdot v = v$. 

Note that $\mathrm{QVal}_{X}^{G}$ can be seen as a subset of $N\otimes \mathbb{Q}$  (see also \cite[Section 3.1]{brion1997varietes}). We just need to show that for any $B$-eigenvector $f_{\chi}\in \mathbb{C}(G/H)$ with eigenvalue $\chi$, we have $v(f_{\chi}) = (\gamma\cdot v)(f)$. This follows from $(\gamma\cdot v)(f_{\chi}) = v(\gamma^{-1}\cdot f_{\chi}) = v(\theta_{\chi}(\gamma^{-1})f_{\chi}) = v(f_{\chi})$, where $\theta$ is the injective homomorphism \eqref{Theta_embedding}.
\end{proof}

From now on, we use $T$ to denote $(N_{G}(H)/H)^{0}$. From the above lemma, $\mathrm{DivVal}_{X}^{T\times G}$ is our valuation cone $\mathcal{V}$, and $\mathrm{DivVal}_{X}^{T\times G*} = \mathcal{V}\backslash \{0\}$.

Let $X$ be a $\mathbb{Q}$-Fano spherical variety. As before, we have
\[H^{0}(X,-rK_{X}) \cong \bigoplus_{m\in M\cap r\Delta}V_{2r \rho_{P}+m}.\]

Here again $2\rho_{P}$ is the weight of the preferable $B-$eigenvector $u$, and $\Delta$ is the polytope associated to $u$.

Because $G$ and $T$ commute with each other when they act on $X$, from multiplicity-freeness (\cite[Theorem 2.1.1]{brion1997varietes}) we see that $u$ is also a $T-$eigenvector. Each $V_{2r\rho_{P}+m}$ is generated by the $B-$eigenvector $u^{r}\chi^{m}$, where $\chi^{m}\in \mathbb{C}(X)^{*}$ is the eigenfunction with eigenvalue $m$. Then each element in $V_{2r\rho_{P}+m}$ has the same $T$-eigenvalue as $u^{r}\chi^{m}$ because $u^{r}\chi^{m}$ generates $V_{2r\rho_{P}+m}$ as a $G-$module and $G$ commutes with $T$.

The section $u$ actually has the trivial $T-$eigenvalue. This can be seen by the construction before the Proposition 4.1 in \cite{gagliardi2015gorenstein} where $s$ stands for our eigenvector $u$. It is given at each point as the wedge of some holomorphic vector fields induced by one-parameter subgroups of $G$. Because $T$ and the one-parameter subgroups of $G$ commute, the holomorphic vector fields are all preserved by the action of $T$. Then $u$ is preserved by the action of $T$. Thus $u^{r}\chi^{m}$ has the $T-$eigenvalue $\bar{m}$, where $\bar{m}$ represents the projection of $m$ on $M/\langle \Sigma'\rangle$.

We have two polytopes with respect to the action of $G$, namely the moment polytope $\Delta^{+}$, and the polytope $\Delta$ related to $u$. They are related by $\Delta^{+} = \Delta + 2\rho_{P}$. Because $u$ has trivial $T$-eigenvalue, we just have one polytope with respect to the action of $T$, namely the moment polytope $\Delta_{T}^{+}$. The projection from $M$ to $M/\langle \Sigma'\rangle$ provides us with a projection from $\Delta$ to $\Delta^{+}_{T}$. We assume that we have a strictly positive smooth function $\mathrm{g}$ on $\Delta^{+}_{T}$.

Now for $v\in \mathcal{V}\backslash \{0\}$, we know that $v$ is also $T$-invariant. So we can consider $S_{k}^{\mathrm{g}}(v)$. Using the same argument in Section \ref{section 4}, we see,

\[S_{k}^{\mathrm{g}}(v) = \frac{\sum_{m\in M\cap k\Delta}\mathrm{g}(\frac{\bar{m}}{k})\prod_{\alpha\in R^{+}}\frac{\langle \alpha,2k\rho_{P}+m+\rho \rangle}{\langle \alpha,\rho\rangle}\frac{\langle m,v\rangle + kh_{\mathcal{C}}(v)}{k}}{\sum_{m\in M\cap k\Delta}\mathrm{g}(\frac{\bar{m}}{k})\prod_{\alpha\in R^{+}}\frac{\langle \alpha,2k\rho_{P}+m+\rho \rangle}{\langle \alpha,\rho\rangle}}.\]
Taking limit, we get:
\begin{equation*}
\begin{split}
    S^{\mathrm{g}}(v) &= \frac{\int_{\Delta}\mathrm{g}(\bar{x})P(x)(\langle x,v\rangle+h_{\mathcal{C}}(v))dx}{\int_{\Delta}\mathrm{g}(\bar{x})P(x)dx},\\
    &= h_{\mathcal{C}}(v) + \langle \mathrm{bar}_{\mathrm{DH}}^{\mathrm{g}}(\Delta),v\rangle,
\end{split}
\end{equation*}
where
\[\mathrm{bar}_{\mathrm{DH}}^{\mathrm{g}}(\Delta) = \frac{\int_{\Delta}\mathrm{g}(\bar{x})P(x)xdx}{\int_{\Delta}\mathrm{g}(\bar{x})P(x)dx}.\]
Then eventually:

\begin{thm}\label{theorem 6.1}
Let $(X,-K_{X})$ be a $\mathbb{Q}-$Fano spherical $G$-variety which is an equivariant completion of $G/H$. Let $\Delta$ be the polytope associated to the $G-$action. Let $T = (N_{G}(H)/H)^{0}$ and we have the corresponding polytope $\Delta_{T}^{+}$. Assume $\mathrm{g}$ is a strictly positive smooth function on $\Delta_{T}^{+}$. Then
\begin{equation*}
\delta^{\mathrm{g}}_{G}(X,-K_{X}) = \min_{v\in E} \frac{h_{\mathcal{C}}(v)}{h_{\mathcal{C}}(v) + \langle \mathrm{bar}_{\mathrm{DH}}^{\mathrm{g}}(\Delta),v\rangle}.
\end{equation*}    
where the set $E$ is introduced before the Theorem \ref{thm 3.1}.
\end{thm}

\section{Ding-stability For g-solitons}
In this section we are going to study Ding-stability for $\mathrm{g}-$solitons on a $\mathbb{Q}$-Fano spherical variety.  

First we recall some definitions from \cite{han2023yau, li2021notes}.

\begin{definition}\label{deftf}
Let $(X,-K_{X})$ be a $\mathbb{Q}$-Fano variety. Assume that there is a $T\times G$-action on $X$, where $T$ is a complex algebraic torus and $G$ is a connected reductive algebraic group. A $T\times G$-equivariant test configuration $(\mathcal{X},\mathcal{L})$ of $(X,-K_{X})$ is the following data:
\begin{itemize}
\item $\mathcal{X}$ is a normal variety with a $\mathbb{C}^{*}\times T\times G-$action on it. We have a $\mathbb{C}^{*}$-equivariant flat proper morphism $\pi:\mathcal{X}\rightarrow \mathbb{C}$. The $T\times G-$action only acts on the fibers of $\pi$.
\item $\mathcal{L}$ is $\pi-$ample $\mathbb{Q}-$line bundle on $\mathcal{X}$. There is a $\mathbb{C}^{*}\times T\times G-$action on $\mathcal{L}$ which is a lift of the $\mathbb{C}^{*}\times T\times G$-action on $\mathcal{X}$.
\item Over $\mathbb{C}^{*}$, we have a $\mathbb{C^{*}}\times T\times G$-equivariant isomorphism $(\mathcal{X},\mathcal{L})\times_{\mathbb{C}^{*}}\mathbb{C}\cong \mathbb{C}^{*}\times (X,-K_{X})$, where on the right side $\mathbb{C}^{*}$ only acts on $\mathbb{C}^{*}$ and $T\times G$ only acts on $(X,-K_{X})$.
\end{itemize}

We can twist a $T\times G-$equivariant test configuration as defined by Hisamoto \cite[Section 2.1]{hisamoto2016uniform}:
\begin{itemize}
\item Let $\mathbb{T}$ be the connected part of the center of $G$. Let $N_{\mathbb{T}}$ be the group of one-parameter subgroups of $\mathbb{T}$. For any $\xi\in N_{\mathbb{T}}$, we can define the twist $(\mathcal{X}_{\xi},\mathcal{L}_{\xi})$ of a $T\times G$-equivariant test configuration $(\mathcal{X},\mathcal{L})$ by changing the original $\mathbb{C}^{*}$-action on $(\mathcal{X},\mathcal{L})$ to a new one which is the composition of the original one with the $\mathbb{C}^{*}-$action on $(\mathcal{X},\mathcal{L})$ induced by $\xi$. The pair $(\mathcal{X}_{\xi},\mathcal{L}_{\xi})$ is a new $T\times G$-equivariant test configuration. 
\item For $\xi\in N_{\mathbb{T}}\otimes\mathbb{Q}$, assume that $k\xi \in N_{\mathbb{T}}$. We first use the base change $z\mapsto z^{k}$ to define a $T\times G-$equivariant test configuration $(\mathcal{X}_{k},\mathcal{L}_{k})$. There is an induced $\mathbb{C}^{*}-$action on $(\mathcal{X}_{k},\mathcal{L}_{k})$. Notice that for $\lambda\in \mathbb{C}^{*}$, the $\lambda$-action on $(\mathcal{X}_{k},\mathcal{L}_{k})$ actually commutes with $\lambda^{k}-$action on $(\mathcal{X},\mathcal{L})$ through the pullback morphism between them. Now the one-parameter subgroup corresponding to $k\xi$ gives a $\mathbb{C}^{*}-$action on $(\mathcal{X}_{k},\mathcal{L}_{k})$ through the $\mathbb{T}-$action on it. Then we twist the $\mathbb{C}^{*}-$action on $(\mathcal{X}_{k},\mathcal{L}_{k})$ with this action. The test configuration after twist is thought of as the $k$ times of $(\mathcal{X}_{\xi},\mathcal{L}_{\xi})$.
\end{itemize}

Taking $G$ to be trivial, we get the definition of a $T-$equivariant test configuration. Taking $T$ and $G$ to be trivial, we get the definition of a test configuration. Some important classes of test configurations are:
\begin{itemize}
\item A test configuration is a product if there is a $\mathbb{C}^{*}-$action on $(X,-K_{X})$ and that $(\mathcal{X},\mathcal{L})\cong \mathbb{C}\times (X,-K_{X})$ in a $\mathbb{C}^{*}-$ equivariant way, where the action on the right hand side is given by the multiplication of $\mathbb{C}^{*}$ on $\mathbb{C}$ and the $\mathbb{C}^{*}-$action on $(X,-K_{X})$.
\item A test configuration $(\mathcal{X},\mathcal{L})$ is special if $\mathcal{L}\cong -K_{\mathcal{X}/\mathbb{C}}$ and $\mathcal{X}_{0}$ is $\mathbb{Q}-$Fano.
\end{itemize}
\end{definition}

When we have a test configuration $(\mathcal{X},\mathcal{L})$, we have a corresponding filtration $\mathcal{F}(\mathcal{X},\mathcal{L})$ of $R=\bigoplus_{r\in \mathbb{N}}H^{0}(X,-rK_{X})$. This is for example described in \cite[Section 2.5]{BHJuniform2017}.

For a $\mathbb{Q}-$Fano variety with a $T-$action, we assume as before that there is a strictly positive smooth function $\mathrm{g}$ on $\Delta_{T}^{+}$.
\begin{definition}
Given a $T-$equivariant test configuration $(\mathcal{X},\mathcal{L})$, we have the following non-Archimedean functionals (See \cite[Section 1.5 and Appendix C]{han2023yau}):, \begin{itemize}
\item $\textbf{L}^{\mathrm{NA}}(\mathcal{X},\mathcal{L}) = \mathrm{lct}(\mathcal{X},-(K_{\mathcal{X}}+\mathcal{L});\mathcal{X}_{0}) - 1$,
\item $\textbf{E}_{\mathrm{g}}^{\mathrm{NA}}(\mathcal{X},\mathcal{L}) = S^{\mathrm{g}}(\mathcal{F}(\mathcal{X},\mathcal{L}))$,
\item $\textbf{D}_{\mathrm{g}}^{\mathrm{NA}} = \textbf{L}^{NA} - \textbf{E}_{\mathrm{g}}^{\mathrm{NA}}$,
\item $\boldsymbol{\Lambda}_{\mathrm{g}}^{\mathrm{NA}}(\mathcal{X},\mathcal{L}) = \boldsymbol{\Lambda}^{\mathrm{NA}}(\mathcal{X},\mathcal{L}) = T(\mathcal{F}(\mathcal{X},\mathcal{L}))$,
\item $\textbf{J}_{\mathrm{g}}^{\mathrm{NA}} = \boldsymbol{\Lambda}_{\mathrm{g}}^{\mathrm{NA}}- \textbf{E}_{\mathrm{g}}^{\mathrm{NA}}$,
\item $\textbf{J}_{\mathrm{g,\mathbb{T}}}^{\mathrm{NA}}(\mathcal{X},\mathcal{L}) = \inf_{\xi\in N_{\mathbb{T}}\otimes \mathbb{Q}}\textbf{J}_{\mathrm{g}}^{\mathrm{NA}}(\mathcal{X}_{\xi},\mathcal{L}_{\xi})$.
\end{itemize}
The functional $\textbf{E}_{\mathrm{g}}^{\mathrm{NA}}(\mathcal{X},\mathcal{L})$ is called the $\mathrm{g}-$weighted non-Archimedean Monge-Ampère energy functional. The functional $\textbf{D}_{\mathrm{g}}^{\mathrm{NA}}$ is called the $\mathrm{g}-$weighted Ding functional.
\end{definition}

When $\mathrm{g}=1$, we omit $\mathrm{g}$ so it is compatible with notations in literature. 

\begin{remark}
What is going to be important to us is that if $(\mathcal{X},\mathcal{-K_{\mathcal{X}}})$ is a special $T-$equivariant test configuration, we have $\textbf{D}_{\mathrm{g}}^{\mathrm{NA}}(\mathcal{X},\mathcal{-K_{\mathcal{X}}}) = \beta^{\mathrm{g}}(v) = A_{X}(v) - S^{\mathrm{g}}(v)$, where $v$ is the restriction of $v_{\mathcal{X}_{0}}$ to $\mathbb{C}(X)$. This is for example shown in the proof of \cite[Theorem 1.16]{li2021notes}.   
\end{remark}

Now we can proceed to definitions of stability:
\begin{definition}
Let $(X,-K_{X})$ be a $\mathbb{Q}$-Fano variety. Assume that there is a $T\times G$-action on $X$, where $T$ is a complex algebraic torus and $G$ is a connected reductive algebraic group. We assume that there is a strictly positive smooth function $\mathrm{g}$ on $\Delta_{T}^{+}$.
\begin{itemize}
\item We say that $X$ is $G-$equivariantly $\mathrm{g}-$weighted Ding-semistable if $\textbf{D}_{\mathrm{g}}^{\mathrm{NA}}(\mathcal{X},\mathcal{L})\geq 0$ for every $T\times G-$equivariant test configuration.
\item We say that $X$ is $G-$equivariantly $\mathrm{g}-$weighted Ding-polystable if it is $G-$equivariantly $\mathrm{g}-$weighted Ding-semistable, and we have equality if and only if $(\mathcal{X},\mathcal{L})$ is a product test configuration.
\item We say that $X$ is $G-$uniformly $\mathrm{g}-$weighted Ding-stable if there exists a constant $\gamma>0$ such that $\textbf{D}_{\mathrm{g}}^{\mathrm{NA}}(\mathrm{X},\mathrm{L})\geq \gamma \cdot \textbf{J}_{\mathrm{g,\mathbb{T}}}^{\mathrm{NA}}(\mathcal{X},\mathcal{L})$ for every $T\times G-$equivariant test configuration.
\end{itemize}
\end{definition}

\begin{remark}
Note that $\textbf{J}_{\mathrm{g}}^{\mathrm{NA}}$ and $\textbf{J}^{\mathrm{NA}}$ are bounded by each other. See for example \cite[Lemma 1.4]{li2021notes}. So we can also define $G-$equivariantly $\mathrm{g}-$weighted uniformly Ding-stable with $\textbf{J}^{\mathrm{NA}}$ and $G-$uniformly $\mathrm{g}-$weighted Ding-stable with $\textbf{J}_{\mathbb{T}}^{\mathrm{NA}}$.
\end{remark}

Although not written down explicitly, the following proposition is actually in \cite{li2021notes,han2023yau}.

\begin{prop}
    Let $(X,-K_{X})$ be a $\mathbb{Q}$-Fano variety. Assume that there is a $T\times G$-action on $X$, where $T$ is a complex algebraic torus and $G$ is a connected reductive algebraic group. We assume that there is a strictly positive smooth function $\mathrm{g}$ on $\Delta_{T}^{+}$. Then $X$ is $G-$equivariantly $\mathrm{g}-$weighted Ding-semistable if and only if $\delta_{G}^{\mathrm{g}}(X,-K_{X})\geq 1$, $X$ is $G-$equivariantly $\mathrm{g}-$weighted uniformly Ding-stable if and only if $\delta_{G}^{\mathrm{g}}(X,-K_{X})> 1$.
\end{prop}

The 'only if' part of the proof of the above proposition is in \cite[Theorem 7.8]{han2023yau}. Note that in the proof of \cite[Theorem 7.8]{han2023yau}, the twist can be dropped because of the considerations in \cite[4.1]{li2022g}. The 'if' part is a consequence of \cite[Theorem C.7]{li2021notes}.

Because of the above proposition and Theorem \ref{theorem 6.1}, we immediately get the following result.

\begin{cor}
Let $(X,-K_{X})$ be a $\mathbb{Q}$-Fano spherical $G$-variety which is an equivariant completion of $G/H$. Let $\Delta$ be the polytope associated to the $G-$action. Let $T = (N_{G}(H)/H)^{0}$ and we have the corresponding polytope $\Delta_{T}^{+}$. Assume $\mathrm{g}$ is a smooth positive function on $\Delta_{T}^{+}$. Then the following statements are equivalent:
    \begin{enumerate}
        \item The $\mathrm{g}-$weighted barycenter $bar_{DH}^{\mathrm{g}}(\Delta)$ is in the dual cone of $-\mathcal{V}_{\mathbb{R}}$.
        \item The pair $(X,-K_{X})$ is G-equivariantly $\mathrm{g}$-weighted Ding-semistable.
    \end{enumerate}
\end{cor}

Now we study when a $\mathbb{Q}-$Fano spherical $G$-variety admits a $\mathrm{g}-$soliton. This follows Golota's idea in \cite[Proposition 5.10]{golota2020delta}. 

Let $X$ be a $\mathbb{Q}-$Fano spherical $G$-variety which is an equivariant completion of some $G/H$. From the last section, we have the torus $T = (N_{G}(H)/H)^{0}$. We have the corresponding polytope $\Delta_{T}^{+}$, and $\mathrm{g}$ is a strictly positive smooth function on $\Delta_{T}^{+}$.

In \cite{delcroix2020k} and \cite{delcroix2023uniform}, Delcroix classifies $G$-equivariant test configurations of $(X,-K_{X})$. We follow more closely \cite[Section 4]{delcroix2023uniform} and recall some of the results there. First, we should remark that any $G-$equivariant test configuration is $T\times G-$equivariant (\cite[section 3.4.1]{delcroix2020k}). 

Like in the toric case, a $G$-equivariant test configuration $(\mathcal{X},\mathcal{L})$ of the pair $(X,-K_{X})$ corresponds to a positive rational piecewise linear functions on the moment polytope $\Delta$ associated to the preferable $B-$eigenvector as before. After communications with Delcroix, we find that there is a minor typo in \cite[Theorem 4.1]{delcroix2023uniform}. This piecewise linear function is $g/r$ instead of $g$ in Delcroix's terminology. So let us write the function in this way:

\[\frac{g(x)}{r} = \inf_{(v,s)\in \mathcal{A}} \left(\frac{v(x)}{-s} + \frac{n_{v,s}}{-rs}\right).\]
Here $v\in \mathcal{V}$, $s$ is a negative integer and $\mathcal{A}$ is the set of divisorial valuations on $\mathcal{X}$ provided by the irreducible component of the central fiber $\mathcal{X}_{0}$. Note that every element in $\mathcal{A}$ is $G\times \mathbb{C}^{*}-$invariant. Each $(v,s)$ corresponds to an irreducible component of the central fiber $\mathcal{X}_{0}$ where $v$(respectively $s$) represents the the restriction of the valuation on $\mathbb{C}(X)$ (respectively $\mathbb{C}(t)$) through the inclusion $\mathbb{C}(X)\xhookrightarrow{} \mathbb{C}(\mathcal{X})$(respectively $\mathbb{C}(t)\xhookrightarrow{} \mathbb{C}(\mathcal{X})$). The pair $(v,s)$ is a primitive element in $N\times \mathbb{Z}\cap \mathcal{V}\times \mathbb{Q}$, where $\mathcal{V}\times \mathbb{Q}$ is actually the cone of $G\times \mathbb{C}^{*}-$ invariant valuations of $X\times \mathbb{C}^{*}$. We also recall that $N=\mathrm{Hom}_{\mathbb{Z}}(M,\mathbb{Z})$, where $M$ is the spherical lattice of $X$. 

By descriptions above we see that $(\mathcal{X},\mathcal{L})$ is an integral test configuration, meaning that the scheme theoretic central fiber $\mathcal{X}_{0}$ is integral, if and only if there is only one element in $\mathcal{A}$ and $s = -1$. When $(\mathcal{X},\mathcal{L})$ is integral, the central fiber $\mathcal{X}_{0}$ is a $G\times \mathbb{C}^{*}$-stable subvariety of $\mathcal{X}$, then it is also a spherical $G\times \mathbb{C}^{*}-$variety by \cite[Corollary 2.3.1]{brion1997varietes}. Especially, $\mathcal{X}_{0}$ is normal.  By \cite[Lemma 2.2]{berman2016k}, the central fiber $\mathcal{X}_{0}$ is $\mathbb{Q}-$Gorenstein and $-K_{\mathcal{X}_{0}} = \mathcal{L}|_{\mathcal{X}_{0}}$ is ample. Recall that a $\mathbb{Q}-$Gorenstein spherical variety always has klt singularities (see \cite[Proposition 5.6]{pasquier2017survey}), thus $\mathcal{X}_{0}$ is $\mathbb{Q}-$Fano. Since $\mathcal{L}$ and $-K_{\mathcal{X}/\mathbb{C}}$ are the same over $\mathbb{C}\backslash \{0\}$ as $\mathbb{C}^{*}-$linearized $\mathbb{Q}-$line bundles and  the central fiber consists of only one irreducible component,  we must have $\mathcal{L}\cong -K_{\mathcal{X}/\mathbb{C}}+c\mathcal{X}_{0}$. So we have the special test configuration $(\mathcal{X},K_{\mathcal{X}/\mathbb{C}})$ which differs from $(\mathcal{X},\mathcal{L})$ by a translation (see \cite[Section 6.3]{BHJuniform2017} for the definition of translation). 

In conclusion, the set $\{v\in \mathcal{V}\cap N\}$ is in bijection with the set of $G\times T-$equivariant test configurations of the pair $(X,-K_{X}).$

As mentioned before, the non-Archimedean Ding functional related to the $\mathrm{g}-$soliton takes a very simple form for $T$-equivariant special test configurations:

\begin{align*}
   \mathrm{D}_{\mathrm{g}}(\mathcal{X},-K_{\mathcal{X}}) &= \beta^{g}(v),\\
    &=  A_{X}(v) - S^{g}(v),\\ 
    &= h_{\mathcal{C}}(v) - \frac{\int_{\Delta}\mathrm{g}(\bar{x})P(x)(\langle x,v\rangle+h_{\mathcal{C}}(v))dx}{\int_{\Delta}\mathrm{g}(\bar{x})P(x)dx},\\
    &= -\biggl\langle \frac{\int_{\Delta}\mathrm{g}(\bar{x})P(x)xdx}{\int_{\Delta}\mathrm{g}(\bar{x})P(x)dx},v\biggr\rangle ,\\
    &= -\langle bar_{DH}^{\mathrm{g}}(\Delta),v\rangle.
\end{align*}
Note also that, by \cite[Theorem 4.1]{delcroix2023uniform}, the special test configuration is furthermore a product test configuration if and only if $v\in \mathrm{Lin}(\mathcal{V})\cap N$, where $\mathrm{Lin}(\mathcal{V})$ is the maximal $\mathbb{Q}-$linear subspace contained in $\mathcal{V}$. So, we get:

\begin{prop}\label{prop 7.2}
Let $(X,-K_{X})$ be a $\mathbb{Q}-$Fano spherical $G$-variety which is an equivariant completion of $G/H$. Let $\Delta$ be the polytope associated to the $G-$action. Let $T = (N_{G}(H)/H)^{0}$ and we have the corresponding polytope $\Delta_{T}^{+}$. Assume $\mathrm{g}$ is a strictly positive smooth function on $\Delta_{T}^{+}$. Then $(X,-K_{X})$ is G-equivariantly $\mathrm{g}$-weighted Ding-polystable for special test configurations if and only if $bar_{DH}^{\mathrm{g}}(\Delta)$ is in the relative interior of the dual cone of $-\mathcal{V}_{\mathbb{R}}$.
\end{prop}

Now we can get $T\times G$-uniformly $\mathrm{g}-$weighted Ding-stability for special test configurations from G-equivariantly $\mathrm{g}$-weighted Ding-polystable for special test configurations. We use '$T\times G$-uniformly' here since Delcroix twists by  $N_{T}\otimes \mathbb{Q}$ in his paper \cite{delcroix2023uniform}. 

\begin{prop}\label{prop 7.3}
Assume the setup from the last proposition and also that $(X,-K_{X})$ is G-equivariantly $\mathrm{g}$-weighted Ding-polystable for special test configurations, then $(X,-K_{X})$ is $T\times G$-uniformly $\mathrm{g}-$weighted Ding-stable for special test configurations. 

\begin{proof}
The non-Archimedean functionals for spherical test configurations are described in \cite[Section 5]{delcroix2023uniform}. Following the description of $\textbf{J}^{\mathrm{NA}}$ there, we just need to prove that there is a $\epsilon>0$ such that the following inequality holds for any $v\in \mathcal{V}$:
\[-\langle bar_{DH}^{\mathrm{g}}(\Delta),v\rangle \geq \epsilon \inf_{l\in \mathrm{Lin}(\mathcal{V})}\int_{\Delta}(\max_{\Delta}(v+l) - (v+l))P dx.\]

We denote the weighted barycenter $bar_{DH}^{\mathrm{g}}(\Delta)$ by $b$. By assumption, $b$ is in the relative interior of $-\mathcal{V}_{\mathbb{R}}^{\vee}$. 

We pick a complement $\mathcal{W}$ of $\mathrm{Lin}(\mathcal{V})$ in $N\otimes \mathbb{Q}$. Now any $ v\in \mathcal{V}$ is written uniquely as $v = v_{1}+v_{2}$, where $v_{1}\in \mathrm{Lin}(\mathcal{V})$ and $v_{2}\in \mathcal{W}\cap \mathcal{V}$. The set $\mathcal{W}\cap \mathcal{V}$ is a strictly convex polyhedral cone inside $\mathcal{W}$. We know that $-b$ as a function vanishes on $\mathrm{Lin}(\mathcal{V})$ and is strictly positive on $\mathcal{W}\cap \mathcal{V}\backslash \{0\}$.

Now we endow $M\otimes \mathbb{R}$ and $N\otimes \mathbb{R}$ with inner products so they are dual Euclidean spaces. Then:
\begin{align*}
\inf_{l\in \mathrm{Lin}(\mathcal{V})}\int_{\Delta}(\max_{\Delta}(v+l) - (v+l))P dx
&\leq \int_{\Delta}(\max_{\Delta}v_{2}-v_{2})Pdx,\\
&\leq C\max_{x,y\in \Delta}\langle v_{2},x-y \rangle,\\
&\leq C|v_{2}|\mathrm{diam}(\Delta),
\end{align*}
where $\mathrm{diam}(\Delta)$ is the diameter of $\Delta$.

We have $-\langle b,v\rangle  = -\langle b,v_{2} \rangle $. The set $\{v\in \overline{\mathcal{W}\cap \mathcal{V}}\subset \mathcal{W}\otimes \mathbb{R} \text{ such that }|v| = 1\}$ is compact. The function $-b$ is strictly positive on it because $\overline{\mathcal{W}\cap \mathcal{V}}$ is a polyhedral cone generated by some elements in $\mathcal{W}\cap \mathcal{V}$. So $-b$ is bounded from below by a positive constant on $\{v\in \overline{\mathcal{W}\cap \mathcal{V}}\subset \mathcal{W}\otimes \mathbb{R}||v| = 1\}$. Then we have $-\langle b,v_{2}\rangle \geq C'|v_{2}|$, $C'>0$. 
\end{proof}
\end{prop}

Then we get the following theorem:

\begin{thm}\label{thm7.2}
Let $(X,-K_{X})$ be a $\mathbb{Q}$-Fano spherical $G$-variety which is an equivariant completion of $G/H$. Let $\Delta$ be the polytope associated to the $G-$action. Let $T = (N_{G}(H)/H)^{0}$ and we have the corresponding polytope $\Delta_{T}^{+}$. Assume $\mathrm{g}$ is a smooth positive function on $\Delta_{T}^{+}$. Then the following three statements are equivalent:
    \begin{enumerate}
        \item The $\mathrm{g}-$weighted barycenter $bar_{DH}^{\mathrm{g}}(\Delta)$ is in the relative interior of the dual cone of $-\mathcal{V}_{\mathbb{R}}$.
        \item The pair $(X,-K_{X})$ is G-equivariantly $\mathrm{g}$-weighted Ding-polystable for special test configurations.
        \item There exists a $\mathrm{g}-$soliton on $X$.
    \end{enumerate}
\end{thm}

\begin{proof}
The first two are equivalent because of the proposition \ref{prop 7.2}. We can get (3) from (1) by using Proposition \ref{prop 7.3} and \cite[Theorem 1.7]{han2023yau}. If we have (3), we get (2) by using \cite[Theorem 1.17]{li2021notes} and \cite[Theorem 1.21]{li2021notes}.
\end{proof}

Theorem \ref{thm7.2} is a generalization of Theorem A in \cite{delcroix2020k}. 

As mentioned in the introduction, after completing this paper, we were informed that Theorem \ref{thm7.2} had already been established in \cite{li2022weighted}. 

\begin{remark}
There is another way to get (3) from (2) in the above theorem. By Minimal Model Program developed in \cite{li2014special},\cite{fujita2019valuative}, we actually know that G-equivariantly integral $\mathrm{g}$-weighted Ding-polystability is equivalent to G-equivariantly $\mathrm{g}$-weighted Ding-polystability. The latter is equivalent to reduced uniformly $\mathrm{g}-$weighted stability by \cite[Theorem 1.17]{li2021notes} and \cite[Theorem 1.3]{blum2023existence}. Then reduced uniformly $\mathrm{g}-$weighted stability is equivalent to the existence of the $\mathrm{g}-$soliton by \cite[Theorem 1.21]{li2021notes}. 
\end{remark}

\section{Applications to Sasaki-Einstein metrics}

In \cite[proposition 2]{apostolov2021weighted}, the authors describe an interesting relationship between a special type of $\mathrm{g}-$soliton and the existence of Sasaki-Einstein structure.

\begin{prop}\label{prop7.1}
Let $X$ be a smooth Fano variety of complex dimension $m$. Let $T$ be a complex torus acting on $X$ with the canonical moment polytope $\Delta_{T}^{+}$. Let $l$ be a positive affine-linear function on $\Delta_{T}^{+}$ written as $l(x) = \langle \xi, x\rangle + a$. Let $T^{c}$ be the real compact torus inside $T$. A $T^{c}-$ invariant Kähler metric $\omega\in 2\pi c_{1}(X)$ is an $l^{-(m+2)}$-soliton if and only if the lift $\hat{\xi}$ of $\xi$ to $K_{X}$ via $l$ is the Reeb vector field of a Sasaki-Einstein structure defined on the unit circle bundle inside $K_{X}$ with respect to the hermitian metric on $K_{X}$ with curvature $-\omega$. 
\end{prop}

The authors of \cite{apostolov2021weighted} use this proposition to prove the existence of a Sasaki-Einstein structure on a unit circle bundle associated to $K_{X}$ for a Fano semi-simple principle bundle $X$ with a smooth toric Fano fiber.

We can use their strategy to get a similar result for smooth Fano horospherical manifolds. First we need to introduce some backgrounds here.

Let $G,B$ be as above. Let $U$ be the unipotent radical of $B$. Let $R$ be the set of roots of $(G,T)$ and $R^{+}$ be the set of positive roots associated to $B$. Thus we have the root decomposition:
\[\mathfrak{g} = \mathfrak{t}\oplus\bigoplus_{\alpha\in R}\mathfrak{g}_{\alpha}\]

Let $S$ be the set of simple roots. Let $H$ be a closed subgroup of $G$ that contains the unipotent radical of the opposite Borel subgroup $B^{-}$. Then $H$ is a horospherical subgroup of $G$ and $G/H$ is a horospherical homogeneous space (see \cite[Definition 2.1]{pasquier2008varietes}). The normalizer of $H$, $N_{G}(H)$, is a parabolic subgroup of $G$ containing $B^{-}$ (see \cite[proposition 2.2]{pasquier2008varietes}). Let $L$ be the unique Levi subgroup of $N_{G}(H)$ that contains $T_{\mathrm{max}}.$ Let $Q$ be the opposite parabolic subgroup of $N_{G}(H)$ in $G$ with respect to $L$. The parabolic subgroup $Q$ has the unipotent radical $R_{u}(Q)$. Let's denote $I\subset S$ as the set of $\alpha\in S$ such that $\mathfrak{g}_{\alpha}$ is not in the Lie algebra of $R_{u}(Q)$. If we denote $R_{I}$ as the sub root system of $R$ generated by $I$, then the Lie algebra of the Levi subgroup $L$ is,
\[\mathfrak{l} = \mathfrak{t}\oplus \bigoplus_{\alpha\in R_{I}}\mathfrak{g}_{\alpha}.\]
The morphism $\theta$ introduced in \eqref{Theta_embedding} is an isomorphism (see \cite[Corollary 4.2.1, Corollary 4.3.1]{brion1997varietes}) since $H$ is a horospherical subgroup. 
What makes the combinatorial criterion in Theorem \ref{thm7.2} simpler is the fact that the valuation cone $\mathcal{V}$ is the whole $N\otimes \mathbb{Q}$ (\cite[Corollary 4.2.1, Corollary 4.3.1]{brion1997varietes}), so the dual cone of $-\mathcal{V}_{\mathbb{R}}$ is just $\{0\}$.

Now we let $X$ be a smooth projective Fano completion of $G/H$. In this case, the preferable section $u$ of $-K_{X}$ has weight $2\rho_{Q} = \sum_{\alpha\in R^{+}\backslash R_{I}}\alpha$ (see \cite[Proposition 4.2]{gagliardi2015gorenstein}) and $P(x) = \prod_{\alpha\in R^{+}\backslash R_{I}}\frac{\langle x+2\rho_{Q},\alpha \rangle}{\langle \rho, \alpha \rangle}$ (see \cite[Section 4.2]{brion1989groupe} or \cite[Section 4.3, Section 5.3]{delcroix2020k}).

We shall prove the following theorem:

\begin{thm}\label{sasakihoro}
Let $X$ be a smooth Fano horospherical variety. Then the total space of the canonical bundle with the zero section removed,
\[K_X^\times := K_X \setminus \{0\},\]
admits a Calabi-Yau cone metric, given by a Sasaki-Einstein structure on a unit circle inside $K_{X}$ with respect to some hermitian metric on $K_{X}$.
\end{thm}

\begin{proof}
We follow the argument of \cite{apostolov2021weighted}. Assume that our $X$ has complex dimension $m$ and the complex dimension of $P/H$ is $r$. So $r$ is also the real dimension of $\Delta$. We have:
\begin{align*}
    m = \mathrm{dim}(X) &= \mathrm{dim}(P/H) + \mathrm{dim}(G/P)\\
                    &= r + \# R^{+}\backslash R_{I}.
\end{align*}

Let's consider the set $\Delta^{*} = \{\xi\in N\otimes \mathbb{R}|\langle \xi,x\rangle + 1\geq0 \text{ for } x\in \Delta\}$. This is the so called dual polytope of $\Delta$.  Polytopes like $\Delta^{*}$ play a very important role in classifying Fano horospherical varieties, see \cite{pasquier2008varietes}. 

Because of Theorem \ref{thm7.2} and Proposition \ref{prop7.1}, we just need to find $\xi\in \mathrm{Int}(\Delta^{*})$ such that:
\begin{equation}\label{eq1}
\int_{\Delta}(\langle \xi,x\rangle + 1)^{-m-2}xP(x)dx = 0
\end{equation}

In order to do this, we consider the functional:

\[\xi\in \mathrm{Int}(\Delta^{*})\mapsto \int_{\Delta}(\langle \xi,x\rangle + 1)^{-m-1}P(x)dx .\]

This functional is clearly strictly convex since we can take derivatives twice and see that the Hessian is strictly positive. Now we have to show that this functional is also proper, meaning the functional goes to infinity when $\xi$ goes to some boundary point of $\Delta^{*}$. If this happens, then the functional has a unique minimum point $\xi\in \mathrm{Int}(\Delta^{*})$, then take the derivative we see that $\xi$ satisfies the equation \ref{eq1}. Now let's say that some sequence $\{ \xi_{i} \}\subset \mathrm{Int}(\Delta^{*})$ converges to some boundary point $\xi$. For any small positive $c$, we consider the set $\Delta_{c} = \{x\in \Delta| \langle \xi,x\rangle + 1 \geq c\}$. Clearly $$\int_{\Delta_{c}}(\langle \xi_{i},x\rangle + 1)^{-m-1}P(x)dx\rightarrow \int_{\Delta_{c}}(\langle \xi,x\rangle + 1)^{-m-1}P(x)dx.$$ Making $c$ smaller and smaller we only need to show: 
\[\int_{\Delta}(\langle \xi,x\rangle + 1)^{-m-1}P(x)dx = +\infty\] 
for $\xi\in \partial \Delta^{*}$.

This is rather clear if $P(x)$ is strictly positive on $\Delta$. But $P(x)$ can have zeroes on $\partial \Delta$. Let's explain how to deal with this case by giving first an intuitive argument. Since $\xi\in \partial \Delta^{*}$, we must have $\langle \xi, y\rangle + 1 = 0$ for some vertex $y$ of $\Delta$. Around $y$, the term $(\langle \xi,x\rangle + 1)^{-m-1}$ provides a pole of $(m+1)$ order. The term $P(x)$ provides zero of at most $\# R^{+}\backslash R_{I}$ order. The quotient then has a pole of at least $(r+1)$ order. Let's provide now more details.

Since $\xi\in \partial \Delta^{*}$, we must have $\langle \xi, y\rangle + 1 = 0$ for some vertex $y$ of $\Delta$. Let's denote $J = \{\alpha\in S| \langle y+2\rho_{Q},\alpha \rangle = 0\} $. Note that simple roots in $I$ are all orthogonal to $2\rho_{Q} + \Delta$, so $I\subset J$. We have a sub-root system $R_{J}$ generated by $J$, let's denote the positive part as $R^{+}_{J}$. Clearly if a positive root $\alpha$ satisfies $\langle y+2\rho_{Q},\alpha\rangle = 0$, we must have $\alpha\in R_{J}^{+}$. We can pick a very large positive number $c$ and consider the set $\Delta' = \Delta\cap \bigcap_{\alpha\in R^{+}_{J}\backslash R_{I}}\{c\langle x+2\rho_{Q}, \alpha\rangle \geq \langle x,\xi\rangle + 1\}$.  Clearly $\Delta'$ is just $\Delta\cap \bigcap_{\alpha\in J\backslash I}\{c\langle x+2\rho_{Q}, \alpha\rangle \geq \langle x,\xi\rangle + 1\}$. Now in some region $\Delta''$ inside $\Delta'$ and close enough to $y$ we have that
\[(\langle \xi,x\rangle + 1)^{-m-1}P(x) \geq C(\langle \xi,x\rangle + 1)^{-m-1+\# R^{+}_{J}\backslash R_{I}}\geq C(\langle \xi,x\rangle + 1)^{-r-1},\]
for some positive constant $C$.
Hence we obtain as expected:
\[\int_{\Delta}(\langle \xi,x\rangle + 1)^{-m-1}P(x)dx \geq C\int_{\Delta''} (\langle \xi,x\rangle + 1)^{-r-1}dx = +\infty.\]

\end{proof}

\begin{remark}
    In \cite{nghiem2023spherical}, Nghiem establishes the equivalence between a weighted volume minimization principle and the existence of a conical Calabi–Yau structure on horospherical cones with klt singularities. One can get Theorem \ref{sasakihoro} by using his method. 
\end{remark}

\begin{remark}
Theorem \ref{sasakihoro} is a variant of \cite[Section 7]{MR2399609} in the toric case.
\end{remark}
\bibliographystyle{plainurl}
\bibliography{References}

Chenxi Yin, Département de mathématiques, UQAM, Montréal, Canada

\textit{Email address}: \textbf{yin.chenxi@courrier.uqam.ca/yin.chenxi.mathematics@gmail.com}
\end{document}